\documentclass[11pt]{amsart}
\usepackage{subfig}
\usepackage{hyperref,url}
\usepackage[numbers,sort&compress]{natbib}
\usepackage{amsfonts}
\usepackage{appendix}
\usepackage{mathtools}
\usepackage{amsmath}
\usepackage{amssymb}
\usepackage{amsthm}
\usepackage{latexsym}
\RequirePackage{natbib}
\usepackage{tikz}
\usepackage{pgfplots}
 \usepackage{pgfplotstable}
\usetikzlibrary{arrows,positioning,chains,fit,shapes,calc,snakes}
\newtheorem{theorem}{Theorem}[section]
\newtheorem{lemma}[theorem]{Lemma}

\newtheorem{proposition}[theorem]{Proposition}
\newtheorem{claim}[theorem]{Claim}

\newtheorem{remark}[theorem]{Remark}

\def\PP{{\mathbb P}}
\def\NN{{\mathbb N}}      
\newcommand{\norm}[1]{\left\Vert #1 \right\Vert} 
\DeclareMathOperator{\diam}{diam}

\begin{document}
\bibliographystyle{plainnat}


\title[Sharp convergence rates for averaged nonexpansive maps]{Sharp convergence rates for averaged nonexpansive maps}

\author[M. Bravo]{Mario Bravo}
\address[M.B.]{Universidad de Santiago de Chile, Departamento de Matem\'atica y Ciencia de la Computaci\'on,
Alameda Libertador Bernardo O'higgins 3363, Santiago, Chile }
\email{ \href{mailto:mario.bravo.g@usach.cl}{\nolinkurl{mario.bravo.g@usach.cl}}}
\thanks{The work of Mario Bravo was partially funded by FONDECYT Grant 11151003 and the N\'ucleo Milenio Informaci\'on y Coordinaci\'on en Redes ICM/FIC RC130003.} 

\author[R. Cominetti]{Roberto Cominetti}
\thanks{Roberto Cominetti gratefully acknowledges the support provided by FONDECYT Grant 1130564.}
\address[R.C.]{Universidad Adolfo Ib\'a\~nez, Facultad de Ingenier\'ia y Ciencias,
Diagonal Las Torres 2640, Santiago, Chile} 
\email{ \href{mailto:roberto.cominetti@uai.cl}{\nolinkurl{roberto.cominetti@uai.cl}}}

\begin{abstract}
We establish sharp estimates for the convergence rate of the  Kranosel'ski\v{\i}--Mann fixed point iteration in general normed spaces, and we use them to show that  the optimal constant  of asymptotic regularity is exactly $1/\sqrt{\pi}$. To this end we 
consider  a nested family
of optimal transport problems that provide a recursive bound for the distance between the iterates. We show that these bounds are tight by building a nonexpansive 
map $T:[0,1]^\NN\to[0,1]^\NN$ that attains them with equality, settling a conjecture by Baillon and Bruck.
The recursive bounds are in turn reinterpreted as absorption probabilities for an underlying Markov 
chain which is used to establish the tightness of the constant $1/\sqrt{\pi}$.
 \end{abstract}


\maketitle

\begin{small}
  \\[2ex]
\end{small}

\section{Introduction}

Let $C$ be a bounded closed convex set in a normed space $(X, \norm{\,\cdot\,})$, and
 $T :C \to C$ a nonexpansive map so that $\norm{Tx - Ty}\leq \norm{x-y}$ for all $x,y \in C$.  
Given an initial point $x^0 \in C$ and scalars $\alpha_n \in ]0,1[$, the Kranosel'ski\v{\i}--Mann 
iteration \cite{man,kra} generates a sequence $(x^n)_{n\in\NN}$ that approximates a fixed point 
 of $T$  by the recursive averaging process
\begin{equation}
  \label{eq:Mann} \tag{KM}
  x^{n+1}= (1-\alpha_{n+1})\,x^n + \alpha_{n+1}\, Tx^n.
\end{equation}
 Besides providing a method to compute fixed points and 
to prove their existence, this iteration appears in other contexts such 
as the Prox and Douglas-Rachford algorithms for convex optimization
(see \cite[Bauschke and Combettes]{bc}),
as well as in the discretization of nonlinear semigroups $\frac{dx}{dt}+[I-\!T](x)=0$
(see \cite[Baillon and Bruck]{bb}).

The convergence of the KM iterates towards a fixed point has been established  under various conditions.
In 1955,  Krasnosel'ski\v{\i} \cite{kra} considered the case $\alpha_n\equiv 1/2$ proving that
convergence holds if $T(C)$ is pre-compact and the space $X$ is uniformly convex,
obtaining as a by-product the existence of fixed points. Soon after, Schaefer \cite{sch} proved that 
convergence also holds when $\alpha_n\equiv \alpha$ with $\alpha\in ]0,1[$, while Edelstein \cite{ede} showed
that the uniform convexity of $X$ could be relaxed to strict convexity. The weak convergence of the iterates was 
studied in \cite{rei}, while extensions to hyperbolic metric spaces were considered in \cite{rsh}. The asymptotic 
properties of the iterates, without assuming the existence of fixed points, were studied  in \cite{bor}.

A crucial step in these results was to show that $\|x^n-Tx^n\|$ tends to 0.
This property was singled out by Browder and Petryshyn \cite{bro} under the name
of {\em asymptotic regularity},  proving that it holds for $\alpha_n\equiv \alpha$ 
when $X$ is uniformly convex and Fix$(T)\neq\phi$, without
any compactness assumption.  The extension to non-constant sequences
$\alpha_n$ with $\sum\alpha_n(1\!-\!\alpha_n)=\infty$ was achieved 
in 1972 by Groetsch \cite{gro}. Until then, uniform convexity seemed  
essential to obtain rate of convergence estimates. For instance, in Hilbert spaces, 
Browder and Petryshin \cite{bro2} proved that for  $\alpha_n\equiv\alpha$ one has $\sum\|x^{n+1}-x^n\|^2<\infty$ 
with $\|x^{n+1}-x^n\|$ decreasing, which readily implies $\|x^n-Tx^n\|=o(1/\sqrt{n})$,  as observed 
in \cite[Baillon and Bruck]{bb2}.

In 1976, Ishikawa \cite{ish} proved that strict convexity was not required and
asymptotic regularity holds in every Banach space as soon as $\sum\alpha_n=\infty$ 
and $\alpha_n\leq 1-\epsilon$.  The same result was obtained independently by Edelstein and O'Brien
\cite{eob}, who observed that the limit $\|x^n-Tx^n\|\to 0$ is uniform 
with respect to the initial point $x^0$. Later, Goebel and Kirk \cite{goe} observed that it
is also uniform in $T$,
namely, for each $\epsilon>0$ we have $\|x^n-Tx^n\|\leq \epsilon$ for all $n\geq n_0$, with $n_0$ 
depending on $\epsilon$ and $C$ but independent of $x^0$ and $T$. More recently,
Kohlenbach \cite{ko1} showed that $n_0$ depends on $C$ only through its diameter.

The first rate of convergence in general normed spaces was 
obtained  in 1992 by  Baillon and Bruck \cite{bb2}: for the case $\alpha_n\equiv\alpha$ 
they proved that $\|x^n-Tx^n\|=O(1/\log n)$ and conjectured
a faster rate of $O(1/\sqrt n)$. In fact,  based on numerical experiments
they guessed a metric estimate that would imply all the previous
results, namely, they claimed the existence of a universal constant $\kappa$ such that 
\begin{equation}\label{bnd}
\|x^n-Tx^n\|\leq \kappa\frac{\mbox{diam}(C)}{\sqrt{\sum_{i=1}^n\alpha_i(1\!-\!\alpha_i)}}.
\end{equation}
 In \cite[Baillon and Bruck]{bb} this estimate was confirmed to hold with $\kappa=1/\sqrt{\pi}\sim 0.5642$ for 
constant $\alpha_n\equiv\alpha$, while the extension to arbitrary sequences  $\alpha_n$ was recently achieved 
 in  \cite[Cominetti, Soto and Vaisman]{csv} with exactly the same constant. This bound  not only subsumes the known 
 results but it also provides an explicit estimate of the fixed point residual, making it possible to 
compute the number of iterations required to attain any prescribed accuracy.

The constant $1/\sqrt{\pi}$, which appears somewhat mysteriously in this context,
arises from a subtle connection between the KM iteration and a certain random 
walk over the integers (see \cite{csv}). A natural question is whether this is 
an essential feature, or whether a sharper bound can be proved by other means.
In this direction, when $T$ is an affine map and $\alpha_n\equiv 1/2$ a tight bound with $\kappa=1/\sqrt{2\pi}\sim 0.3989$ was proved in \cite{bb2}, while for general $\alpha_n$'s
the optimal constant  is
$\kappa= \max_{x\geq 0}\sqrt{x} e^{-x}I_0(x)\sim 0.4688$, where 
$I_0(\cdot)$ is a modified Bessel function of the first kind (see \cite{bcv,csv}).
For nonlinear maps, no sharp bound seems to be available.
In this paper we close this gap by proving
\begin{theorem}\label{thm:main}
The  constant $\kappa=1/\sqrt{\pi}$ in the bound \eqref{bnd} is tight. Specifically, for each $\kappa<1/\sqrt{\pi}$ 
there exists a nonexpansive map $T$ defined on the unit cube $C=[0,1]^\NN\subseteq \ell^\infty(\NN)$,
an initial point $x^0\in C$, and a constant sequence $\alpha_n\equiv\alpha$, such that the corresponding 
KM iterates $x^n$ satisfy for some $n\in \NN$
\begin{equation}
  \label{eq:thm1}
\norm{x^n - Tx^n} >\kappa \frac{\diam{(C)}}{\sqrt{\sum_{i=1}^n \alpha_i(1\!-\!\alpha_i)}}.
\end{equation}
\end{theorem}

Note that by rescaling the norm it can always be assumed that $\diam(C)=1$, which we do from now on. 

Our analysis builds upon  previous results from \cite[Baillon and Bruck]{bb2} and \cite[Aygen-Satik]{aig}, 
for which we first establish a link with optimal transport and thence with a related stochastic 
process connecting with the results in \cite[Cominetti, Soto and Vaisman]{csv}.

Since the proof is somewhat involved, we sketch the main steps.
We begin by noting that $x^n-Tx^n=(x^n-x^{n+1})/\alpha_{n+1}$ so we will search for the best
possible estimates for $\|x^n-x^{n+1}\|$. This is achieved in Section \S\ref{S1} by considering a nested 
sequence of optimal transport
problems that provide a recursive bound $d_{mn}$ for the distance between the iterates $\|x^m-x^n\|\leq d_{mn}$. 
These bounds $d_{mn}$ are universal and do not depend either on the operator $T$ or the 
space $X$.  Moreover, it turns out that the $d_{mn}$'s induce a metric on the integers that satisfies a special 4-point 
inequality, which implies that the optimal transports have a simple structure and can be computed by the
{\em inside-out algorithm}, as conjectured in \cite{bb2}. 
We then show that the bounds $d_{mn}$ are the best possible by  constructing a specific nonexpansive 
map $T$ that attains them all as equalities $\|x^m-x^n\|=d_{mn}$, which settles a conjecture formulated in 
\cite{bb2}. 
In Section \S\ref{S2} we re-interpret the recursive bounds $d_{mn}$ as absorption probabilities
for an underlying Markov chain with transition matrix defined from the optimal transports. 
When $\alpha_n\equiv\alpha\geq\frac{1}{2}$ the optimal transports take a particularly simple form
and the induced Markov chain turns out to be asymptotically 
equivalent for $\alpha\sim 1$ to the Markov chain used in 
\cite{csv}. These results are combined in Section \S\ref{S3} in order to prove that the 
constant $\kappa=1/\sqrt{\pi}$ is optimal, establishing Theorem \ref{thm:main}. 
Some final comments are presented in Section \S\ref{S4}, while the more technical 
proofs are left to the Appendix.

\section{Sharp recursive bounds based on optimal transport}\label{S1}
Following \cite[Baillon and Bruck]{bb2}, we look for the sharpest possible estimates $d_{mn}$
for the distance between the KM iterates  $\|x^m\!-\!x^n\|\leq d_{mn}$. To this end, we note that 
$x^n$ is a weighted average between the previous iterate $x^{n-1}$ and its image $Tx^{n-1}$, so 
it follows inductively that $x^n$ is an average of $x^0,Tx^0, Tx^1,\ldots,Tx^{n-1}$. Explicitly, set  $\alpha_0=1$ and
consider the probability distribution $\pi_i^n=\alpha_i \prod_{k=i+1}^n(1\!-\!\alpha_k)$ for $i=0,\ldots,n$.
Then, adopting the convention $Tx^{-1}=x^0$ we have $x^n=\sum_{i=0}^n \pi_i^n \; Tx^{i-1}$.
\begin{figure}[h]
\centering
\begin{tikzpicture}[scale=0.65]
    \begin{axis}[ybar stacked,
        ymin=0,
        xmin=-1.5,
        xmax= 26.5,
        width  = 12cm,
        height = 8cm,
        bar width=9pt,
        xtick = data,
        table/header=false,
        table/row sep=\\,
        xticklabels from table={ 0\\ \  \\ \  \\ \ \\ \ \\ 5\\ \ \\ \ \\ \ \\ \ \\ 10\\ \ \\ \ \\ \ \\ \ \\ 15\\ \ \\ \ \\ \ \\ \ \\20\\ \ \\ \ \\ \ \\ \ \\25\\}{[index]0},
        yticklabels from table={ 0.0\\ 0.0 \\ 0.05\\ 0.1\\ 0.15\\0.2 \\ 0.25\\ 0.3\\}{[index]0},
        enlarge y limits={value=0.2,upper},
    ]
 \node[anchor=south] at (axis cs:20,0.23) {$\pi^m$};
 \node[anchor=south] at (axis cs:25,0.235) {$\pi^n$};

    \addplot[fill=gray] table[x expr=\coordindex,y index=0]{ 
    0.0007\\
    0.0002\\
    0.0003\\
    0.0005\\
    0.0006\\
    0.0007\\
    0.0012\\
    0.0015\\
    0.0019\\
    0.0030\\
    0.0036\\
    0.0049\\
    0.0070\\
    0.0081\\
    0.0099\\
    0.0154\\
    0.0240\\
    0.0266\\
    0.0296\\
    0.0453\\
    0.0553\\
    0.0760\\
    0.1251\\
    0.1397\\
    0.1826\\
    0.2360\\
};
    \addplot[fill=black] table[x expr=\coordindex,y index=0]{
  0.0022\\
   0.0008\\
   0.0008\\
   0.0016\\
   0.0018\\
   0.0023\\
   0.0040\\
   0.0049\\
   0.0059\\
   0.0096\\
   0.0115\\
   0.0155\\
   0.0223\\
   0.0258\\
   0.0314\\
   0.0487\\
   0.0759\\
   0.0838\\
   0.0933\\
   0.1429\\
   0.1747\\
    0\\
    0\\
    0\\
    0\\
    0\\
};

    \end{axis}
\end{tikzpicture}
\vspace{-2ex}
\caption{\label{pis}A plot of $\pi^m$ and $\pi^n$ ($m=20, n=25$).}
\end{figure}
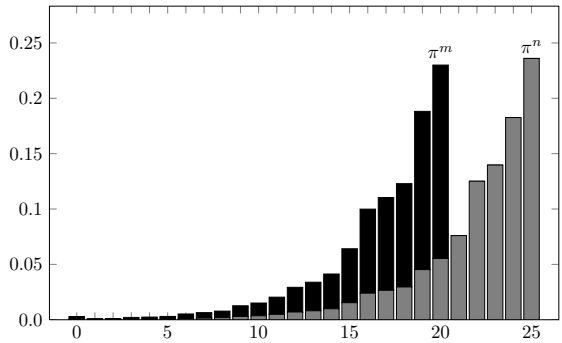

Since a similar formula holds for $x^m$ with $\pi^n$ replaced by $\pi^m$, we may 
try to match both expressions by solving an optimal transport problem between 
$\pi^m$ and $\pi^n$ (see Figure \ref{pis}). Namely, let $\mathcal N=  \{-1,0,1,\dots\}$  
and set $d_{-1,-1}=0$ and $d_{-1,k}=d_{k,-1}=1$ for all $k \in \NN$, and 
then define recursively $d_{mn}$ for $m,n\in\mathcal N$ by  solving a sequence of linear optimal transport problems 
on the complete bipartite graph of Figure \ref{feasible_plan}. Specifically, let us consider 
supplies $\pi_i^m$  at the source nodes $i=0,\ldots,m$ on the left, demands $\pi_j^n$ at the
destination nodes $j=0,\ldots,n$ on the right, unit costs $d_{i-1,j-1}$ on each link $(i,j)$, 
and define
\begin{equation} \label{pmn} \tag{$\mathcal P_{mn}$}
d_{mn}= \mathop{\rm min}_{z \in F^{mn}}C^{mn}(z)=\sum_{i=0}^{m}\sum_{j=0}^{n} z_{ij}\,d_{i-1,j-1}
\end{equation}
where $F^{mn}$ denotes the polytope of feasible transport plans $z$ between $\pi^m$ and $\pi^n$ defined by 
\begin{equation}
  \label{eq:feasible}
  \begin{array}{rcll}
  z_{ij}&\geq& 0&\mbox{ for } i=0,\ldots,m\mbox{ and } j=0,\ldots,n\\[0.5ex]
    \mbox{$\sum_{j=0}^{n}$} \;z_{ij}&=&\pi_i^m&\mbox{ for }\;i=0,\ldots,m\\[0.5ex]
  \mbox{$\sum_{i=0}^{m}$} \;z_{ij}&=&\pi_j^n&\mbox{ for } \; j=0,\ldots,n.
  \end{array}
\end{equation}

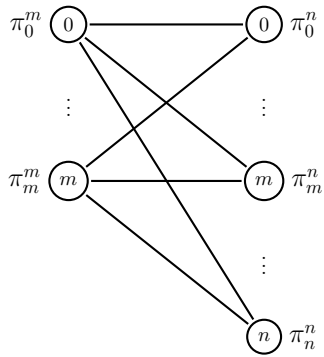
\begin{figure}[h]
\centering
\begin{tikzpicture}[thick, -,shorten >= 1pt,shorten <= 1pt,scale=0.65,every node/.style={scale=0.7}]
\begin{scope}[start chain=going below,node distance=5mm]
 \node[on chain,draw,circle] (i0) [label=left: {\Large $\pi_0^m$}] {$0$};
 \node[on chain] (i1) {$\vdots$};
 \node[on chain,draw,circle] (im) [label=left: {\Large $\pi_m^m$}] {$m$};
\end{scope}
\begin{scope}[xshift=4cm,yshift=0cm,start chain=going below,node distance=5mm]
 \node[on chain,draw,circle] (j0) [label=right: {\Large $\pi_0^n$}] {$0$};
 \node[on chain] (j1) {$\vdots$};
 \node[on chain,draw,circle] (jm) [label=right: {\Large $\pi_m^n$}] {$m$};
 \node[on chain] (j2) {$\vdots$};
 \node[on chain,draw,circle] (jn) [label=right: {\Large $\pi_n^n$}] {$n$};
\end{scope}
\draw (i0) -- (j0);
\draw (i0) -- (jm);
\draw (i0) -- (jn);
\draw (im) -- (j0);
\draw (im) -- (jm);
\draw (im) -- (jn);
\end{tikzpicture}
\caption{\label{feasible_plan}The optimal transport problem \eqref{pmn}.}
\end{figure}

The problems \eqref{pmn} are equivalent to the recursive linear programs considered in \cite[Baillon and Bruck]{bb2}, 
though the formulation in terms of optimal transport is new. The connection between the $d_{mn}$'s and the KM iterates
is as follows.

\begin{theorem}\label{P1} Let $T:C\to C$ be any nonexpansive map  on a convex set $C$ with
$\diam(C)=1$.
Then the KM iterates satisfy $\|x^m-x^n\|\leq d_{mn}$ for all $m,n\in\NN$.
\end{theorem}
\proof 
For each feasible transport plan 
$z\in F^{mn}$ we have
$$
\begin{aligned}
  x^m-x^n &= \sum_{i=0}^{m}\pi_i^m\, Tx^{i-1} - \sum_{j=0}^{n} \pi_j^n\, Tx^{j-1} \\
    &= \sum_{i=0}^{m} \sum_{j=0}^{n} z_{ij}\,Tx^{i-1} -  \sum_{j=0}^{n}\sum_{i=0}^{m}z_{ij}\,Tx^{j-1}\\
            &= \sum_{i=0}^{m} \sum_{j=0}^{n} z_{ij} \left (Tx^{i-1} -Tx^{j-1} \right ).
            \end{aligned}
$$
Now, suppose inductively that the bound $\norm{x^i -x^j} \leq d_{ij}$ 
holds for all $i<m$ and $j<n$ (this is trivially the case for $m=n=1$). Using the triangle inequality, 
the nonexpansivity of $T$, and the fact that $d_{-1,k}=1=\mbox{diam}(C)$ to bound the terms $i=0$ and $j=0$, we get
\begin{equation*}
 \norm{ x^m-x^n} \leq \sum_{i=0}^{m} \sum_{j=0}^{n} z_{ij} \norm{Tx^{i-1} -Tx^{j-1}} \leq \sum_{i=0}^{m} \sum_{j=0}^{n} z_{ij}d_{i-1,j-1}.
\end{equation*}
Minimizing the right hand side over $z\in F^{mn}$ we get  $\norm{ x^m-x^n}\leq d_{mn}$, completing the induction step.
\endproof

We stress that the $d_{mn}$'s are fully determined by the sequence $\alpha_n$ and do not depend on the map $T$,
yet the bound $\|x^m-x^n\|\leq d_{mn}$ is valid for all KM iterates for every nonexpansive map $T:C\to C$ in any normed 
space $X$. Since the proof relies solely on the triangle inequality, a ligitimate question is whether the bounds $d_{mn}$ could be improved 
by using more sophisticated arguments. We will show that these bounds are in fact the best possible, as conjectured
in  \cite[Baillon and Bruck]{bb2}. 

Before proceeding with the proof we make some observations regarding the recursive bounds $d_{mn}$.
Firstly, note that problem \eqref{pmn} is symmetric, so that
$d_{mn}=d_{nm}$ and therefore it suffices to compute these quantities for $m\leq n$. 
On the other hand, any feasible transport $z$ has non-negative entries that add up to one, so that
$C^{mn}(z)$ is an average of previous values $d_{i-1,j-1}$ and then it follows inductively that 
$d_{mn}\in [0,1]$ for all $m,n\in\mathcal N$. It is also easy to prove inductively that $d_{nn}=0$ for all $n\in\mathcal N$. 

Informally, an optimal transport should satisfy the demand $\pi_j^n$ at each node $j=0,\ldots,n$ from the closest
supply nodes $i=0,\ldots,m$. In particular, for $i\leq m\leq n$ we have $\pi_i^n=\pi_i^m\prod_{k=m+1}^n(1\!-\!\alpha_k)\leq\pi_i^m$
and since $d_{i-1,i-1}=0$, the demand $\pi_i^n$ at destination $i$ should be fulfilled directly from the supply $\pi_i^m$
available at the {\em twin} source node $i$. The proof of these intuitive facts is not so simple, though.
The next Theorem, which is due to \cite[Aygen-Satik]{aig}, summarizes some properties that will be relevant 
in the analysis of the KM iteration. Since the proofs in \cite{aig} are rather 
long and technical, in the Appendix we present much simpler proofs which strongly exploit the interpretation
of $d_{mn}$ in terms of optimal transport.

\begin{theorem}[Aygen-Satik \cite{aig}]\label{t1}\ \\[0.5ex]
{\em (a)} The function $d:\mathcal N \times \mathcal N \to [0,1]$ where $d(m,n)=d_{mn}$ defines a distance over the set $\mathcal N$.\\[0.5ex]
{\em (b)} For $m\leq n$ there exists an optimal transport $z$ for $d_{mn}$ with $z_{ii}=\pi_i^n$  for all $i=0,\ldots,m$ 
and {\em a fortiori} $z_{ij}=0$ for all $j\in\{0,\ldots,m\}, j\neq i$. Such an optimal solution will be called {\em simple}.\\[0.5ex]
{\em (c)} For fixed $m$ we have that $n\mapsto d_{mn}$ is decreasing for $n\leq m$ and  increasing for $n\geq m$.\\[0.5ex]
{\em (d)} If $ \alpha_k\geq 1/2$ for all $k \in \mathbb N$, then for all integers $0\leq i\leq k\leq j\leq l$ we have $d_{il}+d_{kj}\leq d_{ij}+d_{kl}$.
We refer to this as the {\em 4-point inequality}.
\end{theorem}
\begin{proof}  See Appendix \ref{appendixA}.
\end{proof}  

\vspace{-0.51cm}
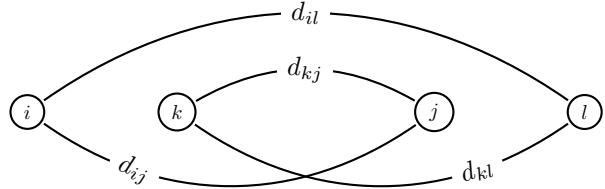
\begin{figure}[h]
\centering
\begin{tikzpicture}[thick,-,shorten >= 1pt,shorten <= 1pt,scale=0.7,every node/.style={scale=0.7}]
\begin{scope}[start chain=going right,node distance=15mm]
 \node[on chain,draw,circle] (m)  {$i$};
 \node[on chain,draw, circle] (k)  {$k$};
 \node[xshift=2cm,on chain,draw, circle] (j)  {$j$};
\node[on chain,draw,circle] (n) {$l$};
\end{scope}
\path[every node/.style={font=\sffamily\small}] (m) edge [bend left = 35]   node [fill=white] {$d_{il}$}  (n) ;
\path[every node/.style={font=\sffamily\small}] (k) edge [bend left]   node  [fill=white]  {$d_{kj}$} (j) ;
\path[every node/.style={font=\sffamily\small}] (k) edge [bend right = 35]  node  [near end,sloped, fill=white] {$d_{kl}$} (n) ;
\path[every node/.style={font=\sffamily\small}] (m) edge [bend right = 35]   node  [near start,sloped,fill=white] {$d_{ij}$} (j) ;

\end{tikzpicture}
\caption{The 4-point inequality.\label{4-point}}
\end{figure}

\begin{remark}\label{inside-out}
{\em Although our proof of the 4-point inequality (d) works only for $\alpha_k\geq 1/2$ (which suffices for our purposes), 
it is worth noting that  in \cite{aig} it is proved for all $\alpha_k\in\; ]0,1[$. A relevant consequence of this property 
is the existence of an optimal transport in which the flows do not cross. 
Namely, if $z$ is a transport plan  such that $z_{ij}>0$ and $z_{kl}>0$ with $i<k<j< l$, then
we can reduce the cost  by removing $\varepsilon=\min\{z_{ij},z_{kl}\}$ from these lower arcs  (see Figure \ref{4-point}) 
while augmenting the uncrossed flows $z_{il}$ and $z_{kj}$ by the same amount. 
This {no flow-crossing} property justifies  in turn the {\em inside-out algorithm} 
which was conjectured in \cite{bb2} to solve $(\mathcal P_{mn})$: set $z_{ii}=\pi_i^n$
for all $i=0,\ldots,m$ and then consider
sequentially the demands $\pi_j^n$  for $j=m+1,\ldots,n$ and fulfill them  from the closest supply 
nodes starting with  $i=m$ and moving towards smaller indices $i=m-1,\ldots,0$ 
when the residual supply $\pi_i^m-\pi_i^n$ of node $i$ is exhausted.
}
\end{remark}

Let us proceed to show that the bounds $d_{mn}$ are sharp. The proof exploits
properties (a), (b), (c) to construct a nonexpansive map $T$ for which the KM iterates satisfy 
$\|x^m-x^n\|=d_{mn}$ for all $m,n$. Property (d) will only be used later in order to show that the constant 
$\kappa=1/\sqrt{\pi}$ in \eqref{bnd} is sharp.


\begin{theorem}\label{teoint}
Let a sequence $\alpha_n\in\; ]0,1[$ be given and consider the corresponding bounds $d_{mn}$.
Then there exists a nonexpansive map $T:C\to C$ defined on the unit cube $C=[0,1]^\NN\subseteq\ell^\infty(\NN)$
and a corresponding KM sequence for which the equality $\norm{ x^m-x^n}_\infty= d_{mn}$ holds for all $m,n\in\NN$.
\end{theorem}

\begin{proof}
The map will be defined on the set $C=[0,1]^Q$ in the space $(\ell^\infty(Q),\|\cdot\|_\infty)$ where $Q$ denotes the
set of all integer pairs $(m,n)$ with $0\leq m\leq n<\infty$, which is isomorphic to $\ell^\infty(\NN)$.

Let us begin by observing that property (b) allows \eqref{pmn} to be reformulated as a min-cost flow on the 
complete graph $G_n$ with node set $\{0,\ldots,n\}$ and costs $d_{i-1,j-1}$ on each arc $(i,j)$,
while considering only the residual demands
$\pi_j^n-\pi_j^m$, namely
$$d_{mn}=\min_{z\geq 0}\left\{\sum_{i,j=0}^{n}\!\!z_{ij}d_{i-1,j-1}:\sum_{i=0}^{n}z_{ij}-\!\sum_{i=0}^{n}z_{ji}
=\pi_j^n\!-\!\pi_j^m\mbox{ for all }j=0,\ldots,n\right\}.\leqno(\mathcal Q_{mn})$$
The residual demands $\pi_j^n-\pi_j^m$ are balanced and add up to zero. They
are negative for $j=0,\ldots,m$ and positive for $j=m+1,\ldots,n$  
(since $\pi_j^n<\pi_j^m$ for $j\leq m$ and $\pi_j^m=0$ for $j >m$).

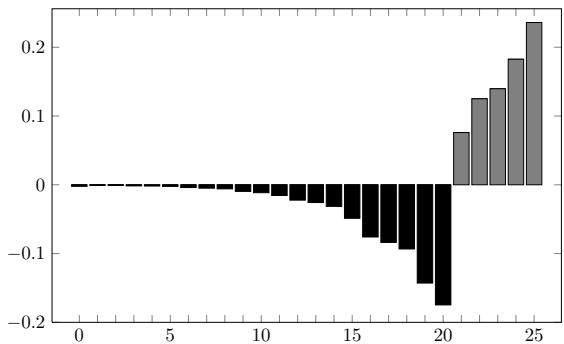
\begin{figure}[h]
\centering
\begin{tikzpicture}[scale=0.65]
    \begin{axis}[
       ybar stacked,
        ymin=-0.2,
        ymax=0.18,
        xmin=-1.5,
        xmax=26.5,
        width  = 12cm,
        height = 8cm,
        bar width=9pt,
        xtick = data,
       xticklabels from table={ 0\\ \  \\ \  \\ \ \\ \ \\ 5\\ \ \\ \ \\ \ \\ \ \\ 10\\ \ \\ \ \\ \ \\ \ \\ 15\\ \ \\ \ \\ \ \\ \ \\20\\ \ \\ \ \\ \ \\ \ \\25\\}{[index]0},
       table/header=false,
        table/row sep=\\,
        enlarge y limits={value=0.2,upper},
    ]
 \addplot[fill=gray] table[x expr=\coordindex,y index=0]{ 
    0\\
    0\\
    0\\
    0\\
    0\\
    0\\
    0\\
    0\\
    0\\
    0\\
    0\\
    0\\
    0\\
    0\\
    0\\
    0\\
    0\\
    0\\
    0\\
    0\\
    0\\ 
    0.0760\\
    0.1251\\
    0.1397\\
    0.1826\\
    0.2360\\
    };
    \addplot[fill=black] table[x expr=\coordindex,y index=0]{ 
   -0.0022\\
   -0.0008\\
   -0.0008\\
   -0.0016\\
   -0.0018\\
   -0.0023\\
   -0.0040\\
   -0.0049\\
   -0.0059\\
   -0.0096\\
   -0.0115\\
   -0.0155\\
   -0.0223\\
   -0.0258\\
   -0.0314\\
   -0.0487\\
   -0.0759\\
   -0.0838\\
   -0.0933\\
   -0.1429\\
   -0.1747\\
    0\\
    0\\
    0\\
    0\\
    0\\
    };
      \end{axis}
\end{tikzpicture}
\caption{\label{pis2} Residual demands $\pi^n-\pi^m$ ($m=20, n=25$).}
\end{figure}


\vspace{2ex}
Now, for each $(m,n)\in Q$ let  $z^{mn}$  be an optimal solution for $(\mathcal Q_{mn})$ 
and $u^{mn}$ optimal for the dual problem 
$$\max_{u}\left\{\sum_{j=0}^{n}(\pi_j^n\!-\!\pi_j^m)u_j:u_j\leq u_i+d_{i-1,j-1}\mbox{ for all $i,j=0,\ldots,n$}\right\}\leqno (\mathcal D_{mn})$$
so that by complementary slackness we have
\begin{equation}\label{cs}
z_{ij}^{mn}(u_j^{mn}\!-u_i^{mn})=z^{mn}_{ij}d_{i-1,j-1}\mbox{ for all $i,j=0,\ldots,n$}.
\end{equation}

Since $(\mathcal D_{mn})$ is invariant by addition of constants
we can fix the value of one component, say $u^{mn}_0=0$, from which 
we get $u^{mn}_j\in [0,1]$ for all $j=0,\ldots,n$.
Indeed, feasibility implies $u^{mn}_j\leq u^{mn}_0+d_{-1,j-1}=1$.
To prove non-negativity we note that $\pi_0^m>\pi_0^n$ so we
can find $i>0$ with $z^{mn}_{0i}>0$ and complementarity 
gives $u^{mn}_i=u^{mn}_0+d_{-1,i-1}=1$, which combined with feasibility 
and $d_{i-1,j-1}\leq 1$ yields $0\leq u^{mn}_i-d_{i-1,j-1}\leq u^{mn}_j$.

Note that the constraints in $(\mathcal D_{mn})$ can be written as
\begin{equation}\label{modulo}
|u^{mn}_j-u^{mn}_i|\leq d_{i-1,j-1}\mbox{ for all }i,j=0,\ldots,n.
\end{equation}
If we extend the definition of $u^{mn}$ by setting $u^{mn}_j=u^{mn}_n$ for $j>n$, 
the monotonicity in Theorem \ref{t1}(c) implies that \eqref{modulo} remains valid for all $i,j\in\NN$.

Now, let us consider the sequence $y^0,y^1,\ldots,y^{k}, \ldots\in [0,1]^Q$ defined as
\begin{equation}\label{yyy}
y^k=(u^{mn}_{k+1})_{mn\in Q}
\end{equation}
and build $x^k\in [0,1]^Q$ starting from $x^0=0\in[0,1]^Q$ by the recursion
\begin{equation}\label{rec}
x^k=(1-\alpha_k)x^{k-1}+\alpha_ky^{k-1}\quad \quad k\geq 1.
\end{equation}
If we define the operator $T$ along this sequence by setting
\begin{equation}\label{map}
Tx^k=y^k\quad \quad k=0,1,\ldots
\end{equation}
then \eqref{rec} corresponds exactly to the Krasnoselskii-Mann iterates for this map.

We will show that $\|x^m-x^n\|_\infty=d_{mn}=\|y^m-y^n\|_\infty$ for all $(m,n)\in Q$.
This implies that $T$ is an isometry and therefore, since $\ell^\infty(Q)$ is an hyperconvex 
space, by the Aronszajn-Panitchpakdi Theorem \cite{aro} it can be extended to a nonexpansive map 
$T:[0,1]^Q\to [0,1]^Q$, which completes the proof.

It remains to show that $\|x^m-x^n\|_\infty=d_{mn}=\|y^m-y^n\|_\infty$.
Since $x^k$ are the Krasnoselskii-Mann iterates for the map $x^k\mapsto Tx^k=y^k$,
as in Theorem \ref{P1} we have
\begin{equation}\label{diferencia}
x^m-x^n=\sum_{i=0}^{m}\sum_{j=0}^{n}z^{mn}_{ij}[y^{i-1}-y^{j-1}].
\end{equation}
From the definition of $y^k$ and since \eqref{modulo} holds for all $i,j\in\NN$, it follows that 
$$\|y^{i-1}-y^{j-1}\|_\infty=\|(u^{m'n'}_i-u^{m'n'}_j)_{m'n'\in Q}\|_\infty\leq d_{i-1,j-1}$$
and then the triangle inequality implies
\begin{equation}\label{siete}
\|x^m-x^n\|_\infty\leq\sum_{i=0}^{m}\sum_{j=0}^{n}z^{mn}_{ij}\|y^{i-1}-y^{j-1}\|_\infty
\leq \sum_{i=0}^{m}\sum_{j=0}^{n}z^{mn}_{ij}d_{i-1,j-1}=d_{mn}.
\end{equation}
On the other hand, if we consider only the $mn$-coordinate in \eqref{diferencia}
we have
 \begin{eqnarray*}
 x_{mn}^m-x_{mn}^n&=&\sum_{i=0}^{m}\sum_{j=0}^{n}z^{mn}_{ij}[y_{mn}^{i-1}-y_{mn}^{j-1}]\\
 &=& \sum_{i=0}^{m}\sum_{j=0}^{n}z^{mn}_{ij}[u^{mn}_i-u^{mn}_j]\\
 &=&\sum_{i=0}^{m}\sum_{j=0}^{n}z^{mn}_{ij}d_{i-1,j-1}=d_{mn}
 \end{eqnarray*}
 where the last equality follows from the complementary slackness \eqref{cs}.
 
 This combined with \eqref{siete} yields $\|x^m-x^n\|_\infty=d_{mn}$. Going back to
 \eqref{siete}, it follows that $\|y^{i-1}\!-y^{j-1}\|_\infty=d_{i-1,j-1}$ as soon as $z^{mn}_{ij}>0$
 for at least one pair $(m,n)\in Q$. 
 Now, this strict inequality holds for 
 $m=j-1$ and $n=j$: the optimal transport in that case is given by $z^{j-1,j}_{ii}\!=\pi_i^j$ 
 and $z^{j-1,j}_{ij}\!=\pi_i^{j-1}\!-\pi_i^{j}>0$  for all $i=0,\ldots,j\!-\!1$. Therefore 
 $\|y^{i-1}\!-y^{j-1}\|_\infty=d_{i-1,j-1}=\|x^{i-1}\!-x^{j-1}\|_\infty$, 
 which establishes the isometry property and completes the proof.
 \end{proof}

\section{A related stochastic process}\label{S2}
For the rest of the analysis we will exploit a probabilistic interpretation of the bounds $d_{mn}$.
To this end, let us fix simple optimal transports $z^{mn}$ for \eqref{pmn} and consider them as transition 
probabilities for a Markov chain $\mathcal D$ 
 interpreted as a race between a fox and a hare:
a hare located at $m$ is being chased by a fox located at $n$ and they sequentially jump 
from $(m,n)$ to a new state $(i-1,j-1)$ with probability $z^{mn}_{ij}$.
If the process ever attains a state $(k,k)$ on the diagonal, then with probability one it will move to a new 
diagonal state $(i-1,i-1)$ so we can represent all diagonal nodes with an absorbing state $f$ meaning that the fox captures the hare. Otherwise the process remains in the
region $m<n$ and eventually reaches a state $(-1,k)$ with $k\geq 0$, in which case the hare can be said to escape by reaching the burrow at position $-1$.
We represent all these latter nodes by an absorbing state $h$ meaning that the hare 
escapes. Hence, the state space for the chain  is
$\mathcal S=\{(m,n):0\leq m<n\}\cup\{h,f\}$ (see Figure \ref{proc}). Note that since $z^{mn}_{ij}=0$ for $j<i$ the process remains in $\mathcal S$.

   \begin{figure}[ht]
\centering
 \begin{tikzpicture}[scale=0.43,every node/.style={scale=0.6}]
  \draw[step=1cm,gray,very thin] (-2,-2) grid (10,10); 
 \draw[thick,->] (0,0) -- (9,0);
 \draw[thick,->] (0,0) -- (0,9);
 \foreach \i in {0,1,...,6}{
	\pgfmathsetmacro{\Start}{\i+1}
 	\foreach \j in {\Start,...,8} 
 	\draw[black,fill=black] (\i,\j) circle (0.1cm);
}
\draw[dashed,-] (-1,-1) -- (8,8);
 \draw[black,fill=black] (7,8) circle (0.1cm);
\node at (5,8.65) {\Large $(m,n)$};
\node[diamond,fill=black] at (-1,5) {};
\node at (-1,5.65) {\Large $h$};
\node[diamond,fill=black] at (5,3) {};
\node at (4.95,2.3) {\Large $f$};
\end{tikzpicture}
 \caption{The state space $\mathcal S$.\label{proc}}
 \end{figure}
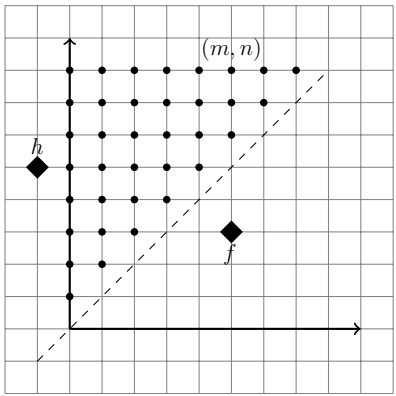
 
 For notational convenience we will write $mn$
instead of $(m,n)$.
All the  states $mn$ are transient and in fact each iteration decreases $m$ by 
at least one unit so that after at most $m$ stages the chain must attain one of the 
absorbing states $f$ or $h$. Now, the basic observation is that the recursion 
$$d_{mn}=\sum_{i=0}^{m}\sum_{j=0}^{n}z^{mn}_{ij}d_{i-1,j-1}$$
characterizes the probability that the hare escapes conditional on the fact that the process
starts in state $mn$,
that is to say $d_{mn}=\PP( \mathcal D \mbox{ attains }h|mn)$.

A similar Markov chain $\mathcal C$ was considered in \cite{csv},
except that the transition probabilities were defined by the following sub-optimal transport 
plans
\begin{equation}\label{ztilde}
\tilde z^{mn}_{ij}=\left\{
\begin{array}{cl}
\pi_i^m\pi_j^n&\mbox{for $i=0,\ldots,m$ and $j=m+1,\ldots,n$}\\
\pi_i^n&\mbox{for $i=j=0,\ldots,m$}\\
0&\mbox{otherwise.}
\end{array}\right.
\end{equation}
These transport plans produced the bounds $c_{mn}=\PP( \mathcal C \mbox{ attains }h|mn)$ which satisfy the analog recursion
$$c_{mn}=\sum_{i=0}^{m}\sum_{j=0}^{n}\tilde z^{mn}_{ij}c_{i-1,j-1}$$
with boundary conditions $c_{-1,-1}=0$ and $c_{-1,k}=c_{k,-1}=1$ for $k\in\NN$.
Since $d_{mn}$ is obtained from the optimal transports while $c_{mn}$ uses a sub-optimal
one, it follows inductively that $d_{mn}\leq c_{mn}$.
Note also that in both recursions the inner sum 
can be replaced by $\sum_{j=m+1}^n$ since the only nonzero flows for $j\leq m$ are 
$z_{jj}$ and $\tilde z_{jj}$ which do not contribute to the sums as they are multiplied
by $d_{j-1,j-1}=0$ and $c_{j-1,j-1}=0$.

Let $\mathbf C$ and  $\mathbf D$ denote the transition
matrices corresponding to $\mathcal C$ and $\mathcal D$, so that
$\mathbf C(s,s')$ and  $\mathbf D(s,s')$ are the transition probabilities 
from state $s\in\mathcal S$ to $s'\in\mathcal S$.  
Recalling that when the process starts from $mn$ it attains one of the absorbing states
after at most $m$ stages, 
the probabilities $d_{mn}$ and $c_{mn}$ can be computed as the $(mn,h)$ entries
of the $m$-th stage transition matrices $\mathbf D^m$ and $\mathbf C^m$, hence
\begin{equation}\label{infnorm}
|c_{mn}-d_{mn}|= |\mathbf D^m(mn,h)-\mathbf C^m(mn,h)|\leq|\!|\!|\mathbf D^m-\mathbf C^m|\!|\!|_\infty
\end{equation}
where $|\!|\!|A|\!|\!|_\infty=\max_{s\in\mathcal S}\sum_{s'\in\mathcal S}|A(s,s')|$ is the operator norm 
induced by $\|\cdot\|_\infty$. 
This estimate, combined with the inside-out algorithm in Remark \ref{inside-out}, allows us to show 
that for $\alpha_n\equiv \alpha\sim 1$ both processes are similar and $|c_{mn}-d_{mn}|$ becomes 
negligible. More precisely,
\begin{proposition} \label{lemma:asymp}
If $\alpha_n\equiv\alpha\in[\frac{1}{2},1[$ then we have $0\leq c_{mn}-d_{mn}\leq 4m(1\!-\!\alpha)^2$ for all $0\leq m\leq n$.
\end{proposition}
\begin{proof} 
See Appendix B.
\end{proof}

\section{Proof of Theorem \ref{thm:main}}\label{S3}
\vspace{2ex}
We now have all the ingredients to prove Theorem~\ref{thm:main}.
Recall that the recursive bounds $d_{mn}$ and $c_{mn}$ were introduced with the aim of finding a sharp 
estimate of $\|x^n-Tx^n\|=\|x^{n+1}-x^n\|/\alpha_{n+1}$, for which the relevant terms are $d_{n,n+1}$ and 
$c_{n,n+1}$.
Let us consider a constant sequence $\alpha_n\equiv\alpha$ (with $\alpha$ to 
be chosen later) and denote $c_{mn}(\alpha), d_{mn}(\alpha)$ the corresponding bounds.
Let us also define
$$
\begin{aligned}
\kappa_n(\alpha)&=\sqrt{n\,\alpha(1\!-\!\alpha)}\;d_{n,n+1}(\alpha)/\alpha,\\
\tilde\kappa_n(\alpha)&=\sqrt{n\,\alpha(1\!-\!\alpha)}\;c_{n,n+1}(\alpha)/\alpha.
\end{aligned}
$$ 

Considering the map $T$ and the KM iterates as given by Theorem \ref{teoint} we have
$$\sqrt{n\,\alpha(1\!-\!\alpha)}\;\|x^n-Tx^n\|=\kappa_n(\alpha)$$
so that Theorem \ref{thm:main} will be proved if we show that 
$\kappa_n(\alpha)$ is arbitrarily close to $1/\sqrt{\pi}$ for 
appropriate values of $n$ and $\alpha$.
We claim that $\kappa_n(\alpha)$ and $\tilde\kappa_n(\alpha)$ 
are very similar for $\alpha\sim 1$ and close to $1/\sqrt{\pi}$ for large $n$. As an illustration, Figure \ref{comp} 
compares these quantities for 4 different $\alpha$'s, where each plot shows the curves $\tilde\kappa_n(\alpha)$ (upper) and 
 $\kappa_n(\alpha)$ (lower) as a function of $n$ for  $n=1,\ldots,300$.
\begin{figure}[h]
\begin{center}
\subfloat[$\alpha=0.5$]{
\begin{tikzpicture}[scale=0.8]
    \begin{axis}[
        table/header=false,
        table/row sep=\\,
        xtick=\empty,
       xmin= -1.0,
       xmax= 301.0,
       ymin= 0.375,
      ymax=0.58,
      ytick={0.4,0.5},
        extra y ticks={0.56419},
       extra y tick labels={$\frac{1}{\sqrt{\pi}}$}
    ]
   \addplot[dashed] coordinates {(0,0.56419) (300,0.56419)};
    \addplot [mark=none] table[x expr=\coordindex,y index=0]{ 
 0.375 \\ 0.441942 \\ 0.473608 \\ 0.492188 \\ 0.504425 \\ 0.513101 \\  0.519574 \\ 0.52459 \\ 0.528591 \\ 0.531857 \\ 
 0.534574 \\ 0.53687 \\ 0.538835 \\ 0.540536 \\ 0.542024 \\ 0.543335 \\ 0.5445 \\ 0.545541 \\ 0.546478 \\ 0.547325 \\ 
 0.548095 \\ 0.548798 \\ 0.549442 \\ 0.550034 \\ 0.55058 \\ 0.551086 \\ 0.551555 \\ 0.551993 \\ 0.5524 \\ 0.552782 \\ 
 0.553139 \\ 0.553475 \\ 0.553791 \\ 0.554089 \\ 0.55437 \\ 0.554636 \\ 0.554888 \\ 0.555127 \\ 0.555354 \\ 0.55557 \\
 0.555776 \\ 0.555972 \\ 0.556159 \\ 0.556338 \\ 0.556509 \\ 0.556673 \\ 0.55683 \\ 0.55698 \\ 0.557125 \\ 0.557263 \\
 0.557397 \\ 0.557525 \\ 0.557649 \\ 0.557768 \\ 0.557883 \\ 0.557994 \\ 0.558101 \\ 0.558204 \\ 0.558304 \\ 0.558401 \\
 0.558494 \\ 0.558585 \\ 0.558672 \\ 0.558757 \\ 0.55884 \\ 0.55892 \\ 0.558997 \\ 0.559073 \\ 0.559146 \\ 0.559217 \\
 0.559286 \\ 0.559353 \\ 0.559419 \\ 0.559482 \\ 0.559544 \\ 0.559605 \\ 0.559664 \\ 0.559721 \\ 0.559777 \\ 0.559832 \\
 0.559885 \\ 0.559937 \\ 0.559987 \\ 0.560037 \\ 0.560085 \\ 0.560132 \\ 0.560179 \\ 0.560224 \\ 0.560268 \\ 0.560311 \\
 0.560353 \\ 0.560394 \\ 0.560435 \\ 0.560474 \\ 0.560513 \\ 0.560551 \\ 0.560588 \\ 0.560625 \\ 0.56066 \\ 0.560695 \\
 0.56073 \\ 0.560763 \\ 0.560796 \\ 0.560828 \\ 0.56086 \\ 0.560891 \\ 0.560922 \\ 0.560952 \\ 0.560981 \\ 0.56101 \\
 0.561039 \\ 0.561067 \\ 0.561094 \\ 0.561121 \\ 0.561147 \\ 0.561173 \\ 0.561199 \\ 0.561224 \\ 0.561249 \\ 0.561273 \\
 0.561297 \\ 0.561321 \\ 0.561344 \\ 0.561367 \\ 0.561389 \\ 0.561411 \\ 0.561433 \\ 0.561454 \\ 0.561475 \\ 0.561496 \\
 0.561516 \\ 0.561537 \\ 0.561556 \\ 0.561576 \\ 0.561595 \\ 0.561614 \\ 0.561633 \\ 0.561651 \\ 0.561669 \\ 0.561687 \\
 0.561705 \\ 0.561722 \\ 0.561739 \\ 0.561756 \\ 0.561773 \\ 0.561789 \\ 0.561806 \\ 0.561822 \\ 0.561837 \\ 0.561853 \\
 0.561868 \\ 0.561884 \\ 0.561899 \\ 0.561913 \\ 0.561928 \\ 0.561942 \\ 0.561957 \\ 0.561971 \\ 0.561985 \\ 0.561998 \\
 0.562012 \\ 0.562025 \\ 0.562038 \\ 0.562051 \\ 0.562064 \\ 0.562077 \\ 0.56209 \\ 0.562102 \\ 0.562114 \\ 0.562126 \\
 0.562138 \\ 0.56215 \\ 0.562162 \\ 0.562174 \\ 0.562185 \\ 0.562196 \\ 0.562208 \\ 0.562219 \\ 0.56223 \\ 0.56224 \\
 0.562251 \\ 0.562262 \\ 0.562272 \\ 0.562283 \\ 0.562293 \\ 0.562303 \\ 0.562313 \\ 0.562323 \\ 0.562333 \\ 0.562343 \\
 0.562352 \\ 0.562362 \\ 0.562371 \\ 0.56238 \\ 0.56239 \\ 0.562399 \\  0.562408 \\ 0.562417 \\ 0.562426 \\ 0.562434 \\
 0.562443 \\ 0.562452 \\ 0.56246 \\ 0.562469 \\ 0.562477 \\ 0.562485 \\ 0.562494 \\ 0.562502 \\ 0.56251 \\ 0.562518 \\
 0.562526 \\ 0.562533 \\ 0.562541 \\ 0.562549 \\ 0.562556 \\ 0.562564 \\ 0.562571 \\ 0.562579 \\ 0.562586 \\ 0.562593 \\
 0.562601 \\ 0.562608 \\ 0.562615 \\ 0.562622 \\ 0.562629 \\ 0.562636 \\ 0.562642 \\ 0.562649 \\ 0.562656 \\ 0.562663 \\
 0.562669 \\ 0.562676 \\ 0.562682 \\ 0.562689 \\ 0.562695 \\ 0.562701 \\ 0.562707 \\ 0.562714 \\ 0.56272 \\ 0.562726 \\
 0.562732 \\ 0.562738 \\ 0.562744 \\ 0.56275 \\ 0.562756 \\ 0.562761 \\ 0.562767 \\ 0.562773 \\ 0.562779 \\ 0.562784 \\
 0.56279 \\ 0.562795 \\ 0.562801 \\ 0.562806 \\ 0.562812 \\ 0.562817 \\ 0.562822 \\ 0.562828 \\ 0.562833 \\ 0.562838 \\
 0.562843 \\ 0.562848 \\ 0.562853 \\ 0.562859 \\ 0.562864 \\ 0.562868 \\ 0.562873 \\ 0.562878 \\ 0.562883 \\ 0.562888 \\
 0.562893 \\ 0.562898 \\ 0.562902 \\ 0.562907 \\ 0.562912 \\ 0.562916 \\ 0.562921 \\ 0.562925 \\ 0.56293 \\ 0.562934 \\
 0.562939 \\ 0.562943 \\ 0.562948 \\ 0.562952 \\ 0.562956 \\ 0.562961 \\ 0.562965 \\ 0.562969 \\ 0.562973 \\ 0.562977 \\
 0.562982 \\ 0.562986 \\ 0.56299 \\ 0.562994 \\ 0.562998 \\ 0.563002 \\ 0.563006 \\ 0.56301 \\ 0.563014 \\ 0.563018 \\
};
  \addplot [mark=none] table[x expr=\coordindex,y index=0]{
0.375\\0.4419417\\0.4668418\\0.4785156\\ 0.4840219\\0.4866107\\0.4876618\\ 0.4878734\\ 0.4876194\\0.4871042\\0.4864480\\ 0.4857208\\ 0.4849652\\0.4842068\\ 0.4834610\\0.4827368\\0.4820393\\0.4813711\\0.4807330\\0.4801250\\0.4795465\\0.4789964\\0.4784734\\0.4779761\\0.4775032\\0.4770532\\0.4766248\\0.4762167\\0.4758276\\0.4754565\\0.4751021\\0.4747635\\0.4744398\\0.4741300\\0.4738333\\0.4735490\\0.4732763\\0.4730145\\ 0.4727631\\0.4725215\\0.4722890\\0.4720654\\0.4718499\\0.4716423\\0.4714421\\0.4712490\\0.4710625\\0.4708823\\0.4707082\\0.4705399\\0.470377 \\0.4702193\\0.4700666\\0.4699187\\0.4697752\\0.4696361\\0.4695011\\0.4693701\\0.4692429\\0.4691193\\0.4689992\\0.4688824\\0.4687688\\0.4686583\\0.4685507\\0.4684460\\0.4683439\\0.4682445\\0.4681477\\0.4680532\\0.4679611\\0.4678712\\0.4677835\\0.4676979\\0.4676144\\0.4675327\\0.4674530\\0.4673751\\0.4672989\\0.4672245\\0.4671516\\0.4670804\\0.4670107\\0.4669425\\0.4668758\\0.4668104\\0.4667464\\0.4666837\\0.4666223\\0.4665621\\0.4665031\\0.4664453\\0.4663886\\0.4663330\\0.4662784\\0.4662249\\0.4661725\\0.4661209\\0.4660704\\0.4660207\\0.465972 \\0.4659241\\0.4658771\\0.4658309\\0.4657856\\0.4657410\\0.4656972\\0.4656541\\0.4656117\\0.4655701\\0.4655292\\0.4654889\\0.4654493\\0.4654103\\0.4653720\\0.4653343\\0.4652971\\0.4652606\\0.4652246\\0.4651892\\0.4651543\\0.4651200\\0.4650862\\0.4650529\\0.4650200\\0.4649877\\0.4649559\\0.4649245\\0.4648935\\0.4648630\\0.4648330\\0.4648034\\0.4647741\\0.4647453\\0.4647169\\0.4646889\\0.4646613\\0.4646340\\0.4646071\\0.4645806\\0.4645544\\0.4645286\\0.4645031\\0.4644780\\0.4644532\\0.4644287\\0.4644045\\0.4643806\\0.4643570\\0.4643337\\0.4643107\\0.4642880\\0.4642656\\0.4642434\\0.4642215\\0.4641999\\0.4641786\\0.4641575\\0.4641366\\0.4641160\\0.4640957\\0.4640755\\0.4640556\\0.4640360\\0.4640166\\0.4639973\\0.4639784\\0.4639596\\0.4639410\\0.4639227\\0.4639045\\0.4638865\\0.4638688\\0.4638512\\0.4638339\\0.4638167\\0.4637997\\0.4637829\\0.4637662\\0.4637498\\0.4637335\\0.4637173\\0.4637014\\0.4636856\\0.4636700\\0.4636545\\0.4636392\\0.4636241\\0.4636091\\0.4635942\\0.4635795\\0.4635650\\0.4635506\\0.4635363\\0.4635222\\0.4635082\\0.4634944\\0.4634806\\0.4634671\\0.4634536\\0.4634403\\0.4634271\\0.4634140\\0.4634011\\0.4633882\\0.4633755\\0.4633629\\0.4633504\\ 0.4633381\\0.4633258\\0.4633137\\0.4633017\\ 0.4632897\\0.4632779\\0.4632662\\0.4632546\\0.4632431\\0.4632317\\0.4632204\\0.4632092\\0.4631981\\0.4631871\\0.4631762\\0.4631654\\0.4631547\\0.4631440\\ 0.4631335\\0.4631230\\0.4631127\\0.4631024\\0.4630922\\0.4630821\\0.4630721\\0.4630621\\0.4630522\\0.4630425\\0.4630328\\0.4630231\\ 0.4630136\\0.4630041\\0.4629947\\0.4629854\\ 0.4629762\\0.462967 \\0.4629579\\0.4629489\\0.4629399\\0.4629310\\0.4629222\\0.4629135\\0.4629048\\0.4628962\\ 0.4628876\\ 0.4628791\\ 0.4628707\\0.4628624\\ 0.4628541\\0.4628458\\0.4628377\\0.4628295\\0.4628215\\0.4628135\\0.4628056\\0.4627977\\0.4627899\\0.4627821\\0.4627744\\0.4627668\\0.4627592\\0.4627517\\0.4627442\\0.4627367\\0.4627294\\0.4627220\\0.4627148\\0.4627075\\0.4627004\\0.4626932\\0.4626862\\0.4626792\\0.4626722\\0.4626653\\0.4626584\\0.4626515\\0.4626448\\0.4626380\\0.4626313\\0.4626247\\0.4626181\\0.4626115\\0.4626050\\0.4625985\\0.4625921\\0.4625857\\0.4625794\\0.4625730\\0.4625668\\0.4625606\\0.4625544\\0.4625482\\
};
    \end{axis}
\end{tikzpicture}}
\subfloat[$\alpha=0.65$]{ 

\begin{tikzpicture}[scale=0.8]
    \begin{axis}[
        table/header=false,
        table/row sep=\\,
        xtick=\empty,
       xmin= -1.0,
      xmax=301.0,
       ymin= 0.375,
      ymax=0.58,
      ytick={0.4,0.5},
        extra y ticks={0.56419},
       extra y tick labels={$\frac{1}{\sqrt{\pi}}$}
    ]
   \addplot[dashed] coordinates {(0,0.56419) (300,0.56419)};
    \addplot [mark=none] table[x expr=\coordindex,y index=0]{ 
 0.368459 \\
 0.437446 \\
 0.470207 \\
 0.489453 \\
 0.502138 \\
 0.511136 \\
 0.517851 \\
 0.523056 \\
 0.527209 \\
 0.5306 \\
 0.533421 \\
 0.535805 \\
 0.537845 \\
 0.539612 \\
 0.541157 \\
 0.542519 \\
 0.543729 \\
 0.544811 \\
 0.545784 \\
 0.546664 \\
 0.547464 \\
 0.548194 \\
 0.548863 \\
 0.549478 \\
 0.550046 \\
 0.550571 \\
 0.551059 \\
 0.551513 \\
 0.551937 \\
 0.552333 \\
 0.552705 \\
 0.553054 \\
 0.553382 \\
 0.553691 \\
 0.553984 \\
 0.55426 \\
 0.554522 \\
 0.55477 \\
 0.555006 \\
 0.555231 \\
 0.555445 \\
 0.555648 \\
 0.555843 \\
 0.556029 \\
 0.556206 \\
 0.556377 \\
 0.55654 \\
 0.556696 \\
 0.556846 \\
 0.55699 \\
 0.557129 \\
 0.557262 \\
 0.557391 \\
 0.557515 \\
 0.557634 \\
 0.557749 \\
 0.557861 \\
 0.557968 \\
 0.558072 \\
 0.558172 \\
 0.558269 \\
 0.558363 \\
 0.558455 \\
 0.558543 \\
 0.558629 \\
 0.558712 \\
 0.558792 \\
 0.558871 \\
 0.558947 \\
 0.559021 \\
 0.559093 \\
 0.559162 \\
 0.55923 \\
 0.559297 \\
 0.559361 \\
 0.559424 \\
 0.559485 \\
 0.559545 \\
 0.559603 \\
 0.559659 \\
 0.559715 \\
 0.559769 \\
 0.559821 \\
 0.559873 \\
 0.559923 \\
 0.559972 \\
 0.56002 \\
 0.560067 \\
 0.560113 \\
 0.560158 \\
 0.560202 \\
 0.560245 \\
 0.560287 \\
 0.560328 \\
 0.560368 \\
 0.560407 \\
 0.560446 \\
 0.560484 \\
 0.560521 \\
 0.560557 \\
 0.560593 \\
 0.560628 \\
 0.560662 \\
 0.560696 \\
 0.560729 \\
 0.560761 \\
 0.560793 \\
 0.560824 \\
 0.560855 \\
 0.560885 \\
 0.560914 \\
 0.560943 \\
 0.560972 \\
 0.561 \\
 0.561027 \\
 0.561054 \\
 0.561081 \\
 0.561107 \\
 0.561133 \\
 0.561158 \\
 0.561183 \\
 0.561207 \\
 0.561231 \\
 0.561255 \\
 0.561278 \\
 0.561301 \\
 0.561324 \\
 0.561346 \\
 0.561368 \\
 0.56139 \\
 0.561411 \\
 0.561432 \\
 0.561452 \\
 0.561473 \\
 0.561493 \\
 0.561512 \\
 0.561532 \\
 0.561551 \\
 0.56157 \\
 0.561588 \\
 0.561607 \\
 0.561625 \\
 0.561643 \\
 0.56166 \\
 0.561677 \\
 0.561695 \\
 0.561711 \\
 0.561728 \\
 0.561744 \\
 0.561761 \\
 0.561777 \\
 0.561792 \\
 0.561808 \\
 0.561823 \\
 0.561839 \\
 0.561854 \\
 0.561868 \\
 0.561883 \\
 0.561897 \\
 0.561912 \\
 0.561926 \\
 0.56194 \\
 0.561953 \\
 0.561967 \\
 0.56198 \\
 0.561993 \\
 0.562007 \\
 0.562019 \\
 0.562032 \\
 0.562045 \\
 0.562057 \\
 0.56207 \\
 0.562082 \\
 0.562094 \\
 0.562106 \\
 0.562118 \\
 0.562129 \\
 0.562141 \\
 0.562152 \\
 0.562163 \\
 0.562175 \\
 0.562186 \\
 0.562196 \\
 0.562207 \\
 0.562218 \\
 0.562228 \\
 0.562239 \\
 0.562249 \\
 0.562259 \\
 0.56227 \\
 0.56228 \\
 0.562289 \\
 0.562299 \\
 0.562309 \\
 0.562319 \\
 0.562328 \\
 0.562337 \\
 0.562347 \\
 0.562356 \\
 0.562365 \\
 0.562374 \\
 0.562383 \\
 0.562392 \\
 0.562401 \\
 0.562409 \\
 0.562418 \\
 0.562427 \\
 0.562435 \\
 0.562443 \\
 0.562452 \\
 0.56246 \\
 0.562468 \\
 0.562476 \\
 0.562484 \\
 0.562492 \\
 0.5625 \\
 0.562507 \\
 0.562515 \\
 0.562523 \\
 0.56253 \\
 0.562538 \\
 0.562545 \\
 0.562553 \\
 0.56256 \\
 0.562567 \\
 0.562574 \\
 0.562581 \\
 0.562588 \\
 0.562595 \\
 0.562602 \\
 0.562609 \\
 0.562616 \\
 0.562623 \\
 0.562629 \\
 0.562636 \\
 0.562642 \\
 0.562649 \\
 0.562655 \\
 0.562662 \\
 0.562668 \\
 0.562674 \\
 0.562681 \\
 0.562687 \\
 0.562693 \\
 0.562699 \\
 0.562705 \\
 0.562711 \\
 0.562717 \\
 0.562723 \\
 0.562729 \\
 0.562734 \\
 0.56274 \\
 0.562746 \\
 0.562752 \\
 0.562757 \\
 0.562763 \\
 0.562768 \\
 0.562774 \\
 0.562779 \\
 0.562785 \\
 0.56279 \\
 0.562795 \\
 0.562801 \\
 0.562806 \\
 0.562811 \\
 0.562816 \\
 0.562821 \\
 0.562826 \\
 0.562832 \\
 0.562837 \\
 0.562842 \\
 0.562846 \\
 0.562851 \\
 0.562856 \\
 0.562861 \\
 0.562866 \\
 0.562871 \\
 0.562875 \\
 0.56288 \\
 0.562885 \\
 0.562889 \\
 0.562894 \\
 0.562898 \\
 0.562903 \\
 0.562908 \\
 0.562912 \\
 0.562916 \\
 0.562921 \\
 0.562925 \\
 0.56293 \\
 0.562934 \\
 0.562938 \\
 0.562942 \\
 0.562947 \\
 0.562951 \\
 0.562955 \\
 0.562959 \\
 0.562963 \\
 0.562967 \\
 0.562971 \\
};
  \addplot [mark=none] table[x expr=\coordindex,y index=0]{
0.3684590\\ 0.4374456\\ 0.4668027\\ 0.4815087\\ 0.4896003\\ 0.4942512\\ 0.4969860\\ 0.4985937\\ 0.4995098\\ 0.4999892\\ 0.5001874\\ 0.5002019\\  
      0.5000959\\ 0.4999109\\ 0.4996749\\ 0.4994069\\ 0.4991200\\ 0.4988233\\ 0.4985230\\ 0.4982234\\ 0.4979276\\ 0.4976376\\ 0.4973546\\ 0.4970797\\ 0.4968132\\ 0.4965554\\ 0.4963065\\ 0.4960662\\ 0.4958345\\ 0.4956113\\ 0.4953961\\ 0.4951888\\ 0.494989\\ 0.4947965\\ 0.4946109\\ 0.4944321\\ 0.4942595\\0.4940931\\ 0.4939325\\ 0.4937774\\ 0.4936277\\ 0.4934830\\ 0.4933431\\ 0.4932079\\ 0.4930770\\ 0.4929504\\ 0.4928279\\ 0.4927091\\ 0.4925941\\ 0.4924825\\ 0.4923744\\ 0.4922695\\ 0.4921676\\ 0.4920688\\ 0.4919727\\ 0.4918794\\ 0.4917888\\ 0.4917006\\ 0.4916148\\ 0.4915314\\ 0.4914502\\ 0.4913711\\0.4912941\\ 0.4912190\\ 0.4911459\\ 0.4910746\\ 0.4910051\\ 0.4909373\\ 0.4908711\\ 0.4908066\\ 0.4907435\\ 0.4906820\\ 0.4906219\\ 0.4905631\\0.4905057\\ 0.4904496\\ 0.4903947\\ 0.4903410\\ 0.4902885\\ 0.4902372\\ 0.4901869\\ 0.4901377\\ 0.4900895\\ 0.4900424\\ 0.4899962\\ 0.4899509\\ 0.4899065\\ 0.4898630\\ 0.4898204\\ 0.4897786\\ 0.4897376\\ 0.4896974\\ 0.4896580\\ 0.4896193\\ 0.4895814\\ 0.4895441\\ 0.4895075\\ 0.4894716\\ 0.4894363\\ 0.4894017\\ 0.4893676\\ 0.4893342\\ 0.4893013\\ 0.4892690\\ 0.4892373\\ 0.4892061\\ 0.4891754\\ 0.4891453\\ 0.4891156\\ 0.4890864\\ 0.4890577\\ 0.4890295\\0.4890017\\ 0.4889743\\ 0.4889474\\ 0.4889209\\ 0.4888949\\ 0.4888692\\ 0.4888439\\ 0.4888190\\ 0.4887945\\ 0.4887703\\ 0.4887465\\ 0.4887230\\ 0.4886999\\ 0.4886772\\ 0.4886547\\ 0.4886326\\ 0.4886108\\ 0.4885893\\ 0.4885681\\ 0.4885472\\ 0.4885266\\ 0.4885062\\ 0.4884862\\ 0.4884664\\ 0.4884469\\ 0.4884276\\ 0.4884086\\ 0.4883899\\ 0.4883714\\ 0.4883531\\ 0.4883351\\ 0.4883173\\ 0.4882997\\ 0.4882824\\ 0.4882653\\ 0.4882484\\ 0.4882317\\ 0.4882152\\ 0.4881989\\ 0.4881828\\ 0.4881669\\ 0.4881512\\ 0.4881357\\ 0.4881204\\ 0.4881052\\ 0.4880902\\ 0.4880755\\ 0.4880608\\ 0.4880464\\ 0.4880321\\ 0.488018\\ 0.4880040\\ 0.4879902\\ 0.4879766\\ 0.4879631\\ 0.4879498\\ 0.4879366\\ 0.4879235\\ 0.4879106\\ 0.4878979\\ 0.4878853\\ 0.4878728\\ 0.4878604\\ 0.4878482\\ 0.4878361\\ 0.4878241\\ 0.4878123\\ 0.4878006\\ 0.4877890\\ 0.4877775\\ 0.4877662\\ 0.4877549\\ 0.4877438\\ 0.4877328\\ 0.4877219\\ 0.4877111\\ 0.4877005\\ 0.4876899\\ 0.4876794\\ 0.4876690\\ 0.4876588\\ 0.4876486\\ 0.4876385\\ 0.4876286\\ 0.4876187\\ 0.4876089\\ 0.4875992\\ 0.4875896\\ 0.4875801\\ 0.4875707\\ 0.4875614\\ 0.4875522\\ 0.4875430\\ 0.4875339\\ 0.4875250\\ 0.4875160\\ 0.4875072\\ 0.4874985\\ 0.4874898\\ 0.48748 12\\ 0.4874727\\ 0.4874643\\ 0.4874559\\ 0.4874476\\ 0.4874394\\ 0.4874313\\ 0.4874232\\ 0.4874152\\ 0.4874073\\ 0.4873994\\ 0.4873916\\ 0.4873839\\ 0.4873762\\ 0.4873686\\ 0.4873611\\ 0.4873536\\ 0.4873462\\ 0.4873388\\ 0.4873316\\ 0.4873243\\ 0.4873172\\ 0.4873101\\ 0.487303\\ 0.487296\\ 0.4872891\\ 0.4872822\\ 0.4872753\\ 0.4872686\\ 0.4872619\\ 0.4872552\\ 0.4872486\\ 0.4872420\\ 0.4872355\\ 0.4872290\\ 0.4872226\\ 0.4872163\\ 0.4872099\\ 0.4872037\\ 0.4871975\\ 0.4871913\\ 0.4871852\\ 0.4871791\\ 0.4871731\\ 0.4871671\\ 0.4871611\\ 0.4871552\\ 0.4871494\\ 0.4871436\\ 0.4871378\\ 0.4871321\\ 0.4871264\\ 0.4871208\\ 0.4871152\\ 0.4871096\\ 0.4871041\\ 0.4870986\\ 0.4870932\\ 0.4870878\\ 0.4870824\\ 0.4870771\\ 0.4870718\\ 0.4870665\\0.4870613\\ 0.4870561\\ 0.487051\\ 0.4870459\\ 0.4870408\\ 0.4870358\\ 0.4870308\\ 0.4870258\\ 0.4870209\\ 0.4870160\\ 0.4870111\\ 0.4870063\\ 0.4870015\\ 0.4869967\\ 0.4869920\\ 0.4869873\\ 0.4869826\\ 0.4869779\\ 0.4869733\\ 0.4869687\\ 0.4869642\\ 0.4869597\\ 0.4869552\\ 0.4869507\\ 0.4869463 \\ 0.4869419\\
};
    \end{axis}
\end{tikzpicture}} 
\end{center}  
\vspace{-0.5ex}

\begin{center}
\hspace*{0.1cm}
\subfloat[$\alpha=0.85$]{

\begin{tikzpicture}[scale=0.8]
    \begin{axis}[
        table/header=false,
        table/row sep=\\,
        xtick=\empty,
         xmin= -1.0,
      xmax=301.0,
         ymin= 0.375,
      ymax=0.58,
      ytick={0.4,0.5},
        extra y ticks={0.56419},
       extra y tick labels={$ \frac{1}{\sqrt{\pi}}$}
    ]
   \addplot[dashed] coordinates {(0,0.56419) (300,0.56419)};
    \addplot [mark=none] table[x expr=\coordindex,y index=0]{ 
 0.311545 \\
 0.392625 \\
 0.435817 \\
 0.46228 \\
 0.47992 \\
 0.492417 \\
 0.501696 \\
 0.508846 \\
 0.51452 \\
 0.519132 \\
 0.522957 \\
 0.52618 \\
 0.528934 \\
 0.531314 \\
 0.533392 \\
 0.535223 \\
 0.536848 \\
 0.5383 \\
 0.539605 \\
 0.540785 \\
 0.541857 \\
 0.542835 \\
 0.54373 \\
 0.544554 \\
 0.545313 \\
 0.546017 \\
 0.546669 \\
 0.547276 \\
 0.547843 \\
 0.548373 \\
 0.548869 \\
 0.549336 \\
 0.549775 \\
 0.550188 \\
 0.550579 \\
 0.550948 \\
 0.551298 \\
 0.55163 \\
 0.551945 \\
 0.552244 \\
 0.55253 \\
 0.552802 \\
 0.553061 \\
 0.55331 \\
 0.553547 \\
 0.553774 \\
 0.553992 \\
 0.5542 \\
 0.554401 \\
 0.554593 \\
 0.554778 \\
 0.554956 \\
 0.555128 \\
 0.555293 \\
 0.555452 \\
 0.555606 \\
 0.555754 \\
 0.555897 \\
 0.556036 \\
 0.55617 \\
 0.5563 \\
 0.556425 \\
 0.556547 \\
 0.556664 \\
 0.556779 \\
 0.55689 \\
 0.556997 \\
 0.557102 \\
 0.557203 \\
 0.557302 \\
 0.557397 \\
 0.557491 \\
 0.557581 \\
 0.55767 \\
 0.557756 \\
 0.557839 \\
 0.557921 \\
 0.558 \\
 0.558078 \\
 0.558153 \\
 0.558227 \\
 0.558299 \\
 0.558369 \\
 0.558438 \\
 0.558505 \\
 0.55857 \\
 0.558634 \\
 0.558697 \\
 0.558758 \\
 0.558818 \\
 0.558876 \\
 0.558933 \\
 0.558989 \\
 0.559044 \\
 0.559098 \\
 0.55915 \\
 0.559202 \\
 0.559252 \\
 0.559302 \\
 0.55935 \\
 0.559398 \\
 0.559444 \\
 0.55949 \\
 0.559535 \\
 0.559579 \\
 0.559622 \\
 0.559664 \\
 0.559706 \\
 0.559747 \\
 0.559787 \\
 0.559826 \\
 0.559865 \\
 0.559903 \\
 0.55994 \\
 0.559977 \\
 0.560013 \\
 0.560048 \\
 0.560083 \\
 0.560117 \\
 0.560151 \\
 0.560184 \\
 0.560217 \\
 0.560249 \\
 0.56028 \\
 0.560311 \\
 0.560342 \\
 0.560372 \\
 0.560402 \\
 0.560431 \\
 0.56046 \\
 0.560488 \\
 0.560516 \\
 0.560543 \\
 0.56057 \\
 0.560597 \\
 0.560623 \\
 0.560649 \\
 0.560675 \\
 0.5607 \\
 0.560724 \\
 0.560749 \\
 0.560773 \\
 0.560797 \\
 0.56082 \\
 0.560843 \\
 0.560866 \\
 0.560888 \\
 0.560911 \\
 0.560933 \\
 0.560954 \\
 0.560975 \\
 0.560996 \\
 0.561017 \\
 0.561038 \\
 0.561058 \\
 0.561078 \\
 0.561098 \\
 0.561117 \\
 0.561136 \\
 0.561155 \\
 0.561174 \\
 0.561192 \\
 0.561211 \\
 0.561229 \\
 0.561247 \\
 0.561264 \\
 0.561282 \\
 0.561299 \\
 0.561316 \\
 0.561333 \\
 0.561349 \\
 0.561366 \\
 0.561382 \\
 0.561398 \\
 0.561414 \\
 0.56143 \\
 0.561445 \\
 0.561461 \\
 0.561476 \\
 0.561491 \\
 0.561506 \\
 0.56152 \\
 0.561535 \\
 0.561549 \\
 0.561563 \\
 0.561577 \\
 0.561591 \\
 0.561605 \\
 0.561619 \\
 0.561632 \\
 0.561645 \\
 0.561659 \\
 0.561672 \\
 0.561685 \\
 0.561697 \\
 0.56171 \\
 0.561723 \\
 0.561735 \\
 0.561747 \\
 0.561759 \\
 0.561772 \\
 0.561783 \\
 0.561795 \\
 0.561807 \\
 0.561819 \\
 0.56183 \\
 0.561841 \\
 0.561853 \\
 0.561864 \\
 0.561875 \\
 0.561886 \\
 0.561896 \\
 0.561907 \\
 0.561918 \\
 0.561928 \\
 0.561939 \\
 0.561949 \\
 0.561959 \\
 0.561969 \\
 0.56198 \\
 0.561989 \\
 0.561999 \\
 0.562009 \\
 0.562019 \\
 0.562028 \\
 0.562038 \\
 0.562047 \\
 0.562057 \\
 0.562066 \\
 0.562075 \\
 0.562084 \\
 0.562093 \\
 0.562102 \\
 0.562111 \\
 0.56212 \\
 0.562129 \\
 0.562137 \\
 0.562146 \\
 0.562155 \\
 0.562163 \\
 0.562171 \\
 0.56218 \\
 0.562188 \\
 0.562196 \\
 0.562204 \\
 0.562212 \\
 0.56222 \\
 0.562228 \\
 0.562236 \\
 0.562244 \\
 0.562252 \\
 0.562259 \\
 0.562267 \\
 0.562274 \\
 0.562282 \\
 0.562289 \\
 0.562297 \\
 0.562304 \\
 0.562311 \\
 0.562318 \\
 0.562326 \\
 0.562333 \\
 0.56234 \\
 0.562347 \\
 0.562354 \\
 0.56236 \\
 0.562367 \\
 0.562374 \\
 0.562381 \\
 0.562387 \\
 0.562394 \\
 0.562401 \\
 0.562407 \\
 0.562414 \\
 0.56242 \\
 0.562427 \\
 0.562433 \\
 0.562439 \\
 0.562445 \\
 0.562452 \\
 0.562458 \\
 0.562464 \\
 0.56247 \\
 0.562476 \\
 0.562482 \\
 0.562488 \\
 0.562494 \\
 0.5625 \\
 0.562506 \\
 0.562511 \\
 0.562517 \\
 0.562523 \\
 0.562529 \\
 0.562534 \\
 0.56254 \\
 0.562545 \\
 0.562551 \\
 0.562556 \\
 0.562562 \\
 0.562567 \\
 };
  \addplot [mark=none] table[x expr=\coordindex,y index=0]{
 0.3115448\\ 0.3926246\\ 0.4356246\\ 0.4616267\\ 0.4785583\\ 0.4901590\\ 0.4984166\\ 0.5044742\\ 0.509028\\ 0.5125215\\ 0.5152479\\ 0.5174070\\ 0.5191386\\ 0.5205426\\ 0.5216920\\ 0.5226408\\ 0.5234300\\ 0.5240906\\ 0.5246470\\ 0.5251178\\ 0.5255182\\ 0.5258599\\ 0.5261527\\ 0.5264042\\ 0.5266209\\ 0.5268080\\ 0.5269698\\ 0.5271100\\ 0.5272315\\ 0.5273369\\ 0.5274283\\ 0.5275076\\ 0.5275764\\ 0.5276359\\ 0.5276873\\ 0.5277316\\ 0.5277696\\ 0.5278021\\ 0.5278297\\ 0.5278531\\ 0.5278726\\ 0.5278887\\ 0.5279018\\ 0.5279122\\ 0.5279203\\ 0.5279262\\ 0.5279303\\ 0.5279326\\ 0.5279335\\ 0.5279330\\ 0.5279313\\ 0.5279286\\ 0.5279248\\ 0.5279203\\ 0.5279150\\ 0.5279090\\ 0.5279024\\ 0.5278953\\ 0.5278878\\ 0.5278798\\ 0.5278714\\ 0.5278628\\ 0.5278538\\ 0.5278446\\ 0.5278352\\ 0.5278256\\ 0.5278159\\ 0.5278060\\ 0.5277960\\ 0.5277859\\ 0.5277758\\ 0.5277656\\ 0.5277553\\ 0.5277451\\ 0.5277348\\ 0.5277245\\ 0.5277142\\ 0.5277039\\ 0.5276937\\ 0.5276835\\ 0.5276733\\ 0.5276632\\ 0.5276531\\ 0.5276431\\ 0.5276331\\ 0.5276232\\ 0.5276133\\ 0.5276035\\ 0.5275938\\ 0.5275841\\ 0.5275746\\ 0.5275651\\ 0.5275556\\ 0.5275463\\ 0.5275370\\ 0.5275279\\ 0.5275188\\ 0.5275097\\ 0.5275008\\ 0.5274919\\ 0.5274832\\ 0.5274745\\ 0.5274659\\ 0.5274573\\ 0.5274489\\ 0.5274405\\ 0.5274322\\ 0.5274240\\ 0.5274159\\ 0.5274079\\ 0.5273999\\ 0.5273920\\ 0.5273842\\ 0.5273765\\ 0.5273689\\ 0.5273613\\ 0.5273538\\ 0.5273464\\ 0.5273391\\ 0.5273318\\ 0.5273246\\ 0.5273175\\ 0.5273105\\ 0.5273035\\ 0.5272966\\ 0.5272897\\ 0.5272830\\ 0.5272763\\ 0.5272697\\ 0.5272631\\ 0.5272566\\ 0.5272502\\ 0.5272438\\ 0.5272375\\ 0.5272313\\ 0.5272251\\ 0.5272190\\ 0.5272129\\ 0.5272069\\ 0.527201\\ 0.5271951\\ 0.5271893\\ 0.5271836\\ 0.5271778\\ 0.5271722\\ 0.5271666\\ 0.5271611\\ 0.5271556\\ 0.5271501\\ 0.5271448\\ 0.5271394\\ 0.5271342\\ 0.5271289\\ 0.5271238\\ 0.5271186\\ 0.5271135\\ 0.5271085\\ 0.5271035\\ 0.5270986\\ 0.5270937\\ 0.5270889\\ 0.5270841\\ 0.5270793\\ 0.5270746\\ 0.5270699\\ 0.5270653\\ 0.5270607\\ 0.5270562\\ 0.5270516\\ 0.5270472\\ 0.5270428\\ 0.5270384\\ 0.5270340\\ 0.5270297\\ 0.5270255\\ 0.5270212\\ 0.5270170\\ 0.5270129\\ 0.5270088\\ 0.5270047\\ 0.5270006\\ 0.5269966\\ 0.5269927\\ 0.5269887\\ 0.5269848\\ 0.5269809\\ 0.5269771\\ 0.5269733\\ 0.5269695\\ 0.5269657\\ 0.5269620\\ 0.5269583\\ 0.5269547\\ 0.5269511\\ 0.5269475\\ 0.5269439\\0.5269404\\ 0.5269369\\ 0.5269334\\ 0.5269299\\ 0.5269265\\ 0.5269231\\ 0.5269198\\ 0.5269164\\ 0.5269131\\ 0.5269098\\ 0.5269066\\ 0.5269033\\ 0.5269001\\ 0.5268969\\ 0.5268938\\ 0.5268906\\ 0.5268875\\ 0.5268844\\ 0.5268814\\ 0.5268783\\ 0.5268753\\ 0.5268723\\ 0.5268693\\ 0.5268664\\ 0.5268635\\ 0.5268606\\ 0.5268577\\ 0.5268548\\ 0.5268520\\ 0.5268492\\ 0.5268464\\ 0.5268436\\ 0.5268408\\ 0.5268381\\ 0.5268354\\ 0.5268327\\ 0.5268300\\ 0.5268274\\ 0.5268247\\ 0.5268221\\ 0.5268195\\ 0.5268169\\ 0.5268144\\ 0.5268118\\ 0.5268093\\ 0.5268068\\ 0.5268043\\ 0.5268018\\ 0.5267994\\ 0.5267969\\ 0.5267945\\ 0.5267921\\ 0.5267897\\ 0.5267873\\ 0.5267850\\ 0.5267826\\ 0.5267803\\ 0.5267780\\ 0.5267757\\ 0.5267734\\ 0.5267711\\ 0.5267689\\ 0.5267667\\ 0.5267644\\ 0.5267622\\ 0.5267601\\ 0.5267579\\ 0.5267557\\ 0.5267536\\ 0.5267514\\ 0.5267493\\ 0.5267472\\ 0.5267451\\ 0.5267430\\ 0.5267410\\ 0.5267389\\ 0.5267369\\ 0.5267349\\ 0.5267329\\ 0.5267309\\ 0.5267289\\ 0.5267269\\ 0.5267249\\ 0.5267230\\ 0.5267210\\ 0.5267191\\ 0.5267172\\ 0.5267153\\ 0.5267134\\ 0.5267115\\ 0.5267097\\ 0.5267078\\ 0.5267060\\ 0.5267041\\ 0.5267023\\ 0.5267005\\ 0.5266987\\ 0.5266969\\ 0.5266951\\ 0.5266934\\ 0.5266916\\ 0.5266899\\ 0.5266881\\ 0.5266864\\   
};
    \end{axis}
\end{tikzpicture}}
\hspace{0.05cm}
 \subfloat[$\alpha=0.99$]{

\begin{tikzpicture}[scale=0.8]
    \begin{axis}[
        table/header=false,
        table/row sep=\\,
        xtick=\empty,
       xmin= -1.0,
      xmax=301.0,
       ymin= 0.375,
      ymax=0.58,
      ytick={0.4,0.5},
        extra y ticks={0.56419},
       extra y tick labels={$ \frac{1}{\sqrt{\pi}}$}
    ]
   \addplot[dashed] coordinates {(0,0.56419) (300,0.56419)};
    \addplot [mark=none] table[x expr=\coordindex,y index=0]{
 0.0985137 \\
 0.137954 \\
 0.167319 \\
 0.191347 \\
 0.211898 \\
 0.229938 \\
 0.246046 \\
 0.260607 \\
 0.27389 \\
 0.286094 \\
 0.297371 \\
 0.30784 \\
 0.317597 \\
 0.326721 \\
 0.335279 \\
 0.343324 \\
 0.350905 \\
 0.358062 \\
 0.364831 \\
 0.371243 \\
 0.377326 \\
 0.383105 \\
 0.3886 \\
 0.393832 \\
 0.398819 \\
 0.403577 \\
 0.40812 \\
 0.412461 \\
 0.416614 \\
 0.420588 \\
 0.424395 \\
 0.428044 \\
 0.431544 \\
 0.434902 \\
 0.438127 \\
 0.441226 \\
 0.444204 \\
 0.447069 \\
 0.449825 \\
 0.452479 \\
 0.455035 \\
 0.457498 \\
 0.459872 \\
 0.462162 \\
 0.464371 \\
 0.466503 \\
 0.468561 \\
 0.47055 \\
 0.472471 \\
 0.474328 \\
 0.476124 \\
 0.477861 \\
 0.479541 \\
 0.481168 \\
 0.482743 \\
 0.484268 \\
 0.485746 \\
 0.487178 \\
 0.488567 \\
 0.489913 \\
 0.491219 \\
 0.492486 \\
 0.493715 \\
 0.494909 \\
 0.496068 \\
 0.497193 \\
 0.498287 \\
 0.49935 \\
 0.500383 \\
 0.501388 \\
 0.502364 \\
 0.503315 \\
 0.504239 \\
 0.505139 \\
 0.506015 \\
 0.506868 \\
 0.507699 \\
 0.508508 \\
 0.509297 \\
 0.510065 \\
 0.510814 \\
 0.511544 \\
 0.512257 \\
 0.512951 \\
 0.513629 \\
 0.51429 \\
 0.514936 \\
 0.515566 \\
 0.516182 \\
 0.516783 \\
 0.51737 \\
 0.517943 \\
 0.518504 \\
 0.519051 \\
 0.519587 \\
 0.520111 \\
 0.520623 \\
 0.521123 \\
 0.521613 \\
 0.522093 \\
 0.522562 \\
 0.523021 \\
 0.523471 \\
 0.523911 \\
 0.524343 \\
 0.524765 \\
 0.525179 \\
 0.525584 \\
 0.525981 \\
 0.526371 \\
 0.526753 \\
 0.527127 \\
 0.527494 \\
 0.527854 \\
 0.528207 \\
 0.528554 \\
 0.528894 \\
 0.529227 \\
 0.529555 \\
 0.529876 \\
 0.530192 \\
 0.530502 \\
 0.530806 \\
 0.531105 \\
 0.531399 \\
 0.531687 \\
 0.531971 \\
 0.532249 \\
 0.532523 \\
 0.532792 \\
 0.533057 \\
 0.533317 \\
 0.533573 \\
 0.533825 \\
 0.534073 \\
 0.534316 \\
 0.534556 \\
 0.534792 \\
 0.535024 \\
 0.535252 \\
 0.535477 \\
 0.535699 \\
 0.535917 \\
 0.536131 \\
 0.536343 \\
 0.536551 \\
 0.536756 \\
 0.536958 \\
 0.537157 \\
 0.537353 \\
 0.537546 \\
 0.537737 \\
 0.537925 \\
 0.53811 \\
 0.538292 \\
 0.538472 \\
 0.538649 \\
 0.538824 \\
 0.538997 \\
 0.539167 \\
 0.539335 \\
 0.5395 \\
 0.539664 \\
 0.539825 \\
 0.539984 \\
 0.540141 \\
 0.540296 \\
 0.540449 \\
 0.5406 \\
 0.540749 \\
 0.540896 \\
 0.541041 \\
 0.541184 \\
 0.541326 \\
 0.541466 \\
 0.541604 \\
 0.54174 \\
 0.541875 \\
 0.542008 \\
 0.54214 \\
 0.54227 \\
 0.542398 \\
 0.542525 \\
 0.54265 \\
 0.542774 \\
 0.542897 \\
 0.543018 \\
 0.543137 \\
 0.543256 \\
 0.543373 \\
 0.543488 \\
 0.543603 \\
 0.543716 \\
 0.543828 \\
 0.543938 \\
 0.544048 \\
 0.544156 \\
 0.544263 \\
 0.544369 \\
 0.544474 \\
 0.544577 \\
 0.54468 \\
 0.544781 \\
 0.544882 \\
 0.544981 \\
 0.545079 \\
 0.545177 \\
 0.545273 \\
 0.545368 \\
 0.545463 \\
 0.545556 \\
 0.545648 \\
 0.54574 \\
 0.545831 \\
 0.54592 \\
 0.546009 \\
 0.546097 \\
 0.546184 \\
 0.54627 \\
 0.546356 \\
 0.546441 \\
 0.546524 \\
 0.546607 \\
 0.54669 \\
 0.546771 \\
 0.546852 \\
 0.546932 \\
 0.547011 \\
 0.547089 \\
 0.547167 \\
 0.547244 \\
 0.54732 \\
 0.547396 \\
 0.547471 \\
 0.547545 \\
 0.547619 \\
 0.547692 \\
 0.547764 \\
 0.547836 \\
 0.547907 \\
 0.547977 \\
 0.548047 \\
 0.548116 \\
 0.548185 \\
 0.548253 \\
 0.548321 \\
 0.548387 \\
 0.548454 \\
 0.548519 \\
 0.548585 \\
 0.548649 \\
 0.548713 \\
 0.548777 \\
 0.54884 \\
 0.548903 \\
 0.548965 \\
 0.549026 \\
 0.549087 \\
 0.549148 \\
 0.549208 \\
 0.549267 \\
 0.549326 \\
 0.549385 \\
 0.549443 \\
 0.549501 \\
 0.549558 \\
 0.549615 \\
 0.549671 \\
 0.549727 \\
 0.549783 \\
 0.549838 \\
 0.549892 \\
 0.549947 \\
 0.55 \\
 0.550054 \\
 0.550107 \\
 0.550159 \\
 0.550212 \\
 0.550264 \\
 0.550315 \\
 0.550366 \\
 0.550417 \\
 0.550467 \\
 0.550517 \\
 0.550566 \\
 0.550616 \\
 0.550665 \\
 0.550713 \\
 0.550761 \\
 0.550809 \\
 0.550856 \\
 0.550904 \\
 0.55095 \\
 0.550997 \\
 0.551043 \\
 0.551089 \\
 0.551134 \\
 0.551179 \\
 0.551224 \\
 0.551269 \\
};
\addplot [mark=none] table[x expr=\coordindex,y index=0]{
0.0985137\\0.1379539\\0.1673190\\0.1913474\\0.2118984\\0.2299375\\0.2460459\\0.2606067\\0.2738897\\0.2860935\\0.2973702\\0.3078391\\0.3175962\\0.3267204\\0.3352773\\0.3433223\\0.3509027\\0.3580594\\0.3648280\\0.3712398\\0.3773222\\0.3830998\\0.3885945\\0.3938260\\0.3988120\\0.4035686\\0.4081105\\0.4124510\\0.4166024\\0.4205757\\0.4243814\\0.4280289\\0.4315270\\0.4348840\\0.4381073\\0.4412040\\0.4441806\\0.4470434\\0.4497979\\0.4524495\\0.4550032\\0.4574638\\0.4598356\\0.4621228\\0.4643293\\0.4664588\\0.4685147\\0.4705004\\0.4724188\\0.4742730\\0.4760658\\0.4777996\\0.4794771\\0.4811006\\0.4826723\\0.4841943\\0.4856686\\0.4870973\\0.4884820\\0.4898245\\0.4911266\\0.4923897\\0.4936153\\0.4948049\\0.4959599\\0.4970816\\0.4981711\\0.4992297\\0.5002585\\0.5012587\\0.5022312\\0.5031770\\0.5040971\\0.5049924\\0.5058637\\0.5067120\\0.5075380\\0.5083424\\0.5091260\\0.5098896\\0.5106337\\0.5113590\\0.5120662\\0.5127559\\0.5134286\\0.5140848\\0.5147252\\0.5153502\\0.5159603\\0.5165559\\0.5171377\\0.5177058\\0.5182609\\0.5188033\\0.5193334\\0.5198515\\0.5203580\\0.5208534\\0.5213378\\0.5218116\\0.5222752\\0.5227289\\0.5231729\\0.5236075\\0.5240329\\0.5244495\\0.5248576\\0.5252572\\0.5256487\\0.5260324\\0.5264083\\0.5267768\\0.5271380\\0.5274922\\0.5278394\\0.5281800\\0.5285141\\0.5288418\\0.5291634\\0.5294789\\0.5297886\\0.5300926\\0.5303910\\0.5306840\\0.5309718\\0.5312544\\0.5315320\\0.5318047\\0.5320726\\0.5323359\\0.5325947\\0.5328490\\0.533099\\0.5333448\\0.5335865\\0.5338241\\0.5340579\\0.5342878\\0.534514\\0.5347365\\0.5349555\\0.5351710\\0.5353830\\0.5355918\\0.5357973\\0.5359996\\0.5361988\\0.5363949\\0.5365881\\0.5367784\\0.5369658\\0.5371505\\0.5373324\\0.5375116\\0.5376883\\0.5378624\\0.5380340\\0.5382031\\0.5383698\\0.5385342\\0.5386963\\0.5388562\\0.5390138\\0.5391693\\0.5393226\\0.5394739\\0.5396231\\0.5397703\\0.5399156\\0.5400590\\0.5402004\\0.5403401\\0.5404779\\0.5406139\\0.5407482\\0.5408808\\0.5410117\\0.5411410\\0.5412686\\0.5413947\\0.5415192\\0.5416422\\0.5417637\\0.5418837\\0.5420022\\0.5421193\\0.5422351\\0.5423494\\0.5424624\\0.5425741\\0.5426845\\0.5427936\\0.5429015\\0.5430081\\0.5431135\\0.5432177\\0.5433207\\0.5434226\\0.5435233\\0.5436230\\0.5437215\\0.5438189\\0.5439153\\0.5440106\\0.5441049\\0.5441982\\0.5442905\\0.5443818\\0.5444722\\0.5445616\\0.5446501\\0.5447376\\0.5448242\\0.5449100\\0.5449949\\0.5450789\\0.5451620\\0.5452443\\0.5453258\\0.5454065\\0.5454863\\0.5455654\\0.5456437\\0.5457212\\0.5457980\\0.5458741\\0.5459493\\0.5460239\\0.5460978\\0.5461709\\0.5462434\\0.5463152\\0.5463863\\0.5464567\\0.5465265\\0.5465956\\0.5466641\\0.5467320\\0.5467993\\0.5468659\\0.5469319\\0.5469974\\0.5470622\\0.5471265\\0.5471902\\0.5472534\\0.5473159\\0.5473780\\0.5474395\\0.5475004\\0.5475609\\0.5476208\\0.5476802\\0.5477390\\0.5477974\\0.5478553\\0.5479127\\0.5479696\\0.5480261\\0.5480821\\0.5481376\\0.5481926\\0.5482472\\0.5483014\\0.5483551\\0.5484083\\0.5484612\\0.5485136\\0.5485656\\0.5486171\\0.5486683\\0.5487190\\0.5487694\\0.5488193\\0.5488689\\0.5489181\\0.5489669\\0.5490153\\0.5490633\\0.549111\\0.5491583\\0.5492052\\0.5492518\\0.5492981\\0.5493439\\0.5493895\\0.5494347\\0.5494795\\0.5495241\\0.5495683\\0.5496121\\0.5496557\\0.5496989\\0.5497418\\0.5497844\\0.5498267\\0.5498687\\0.5499104\\0.5499517\\0.5499928\\
};
    \end{axis}
\end{tikzpicture} }
\end{center}  
\caption{\label{comp}Comparison of $\kappa_n(\alpha)$ and $\tilde\kappa_n(\alpha)$ for $n=1,\ldots,300$. }
\end{figure}
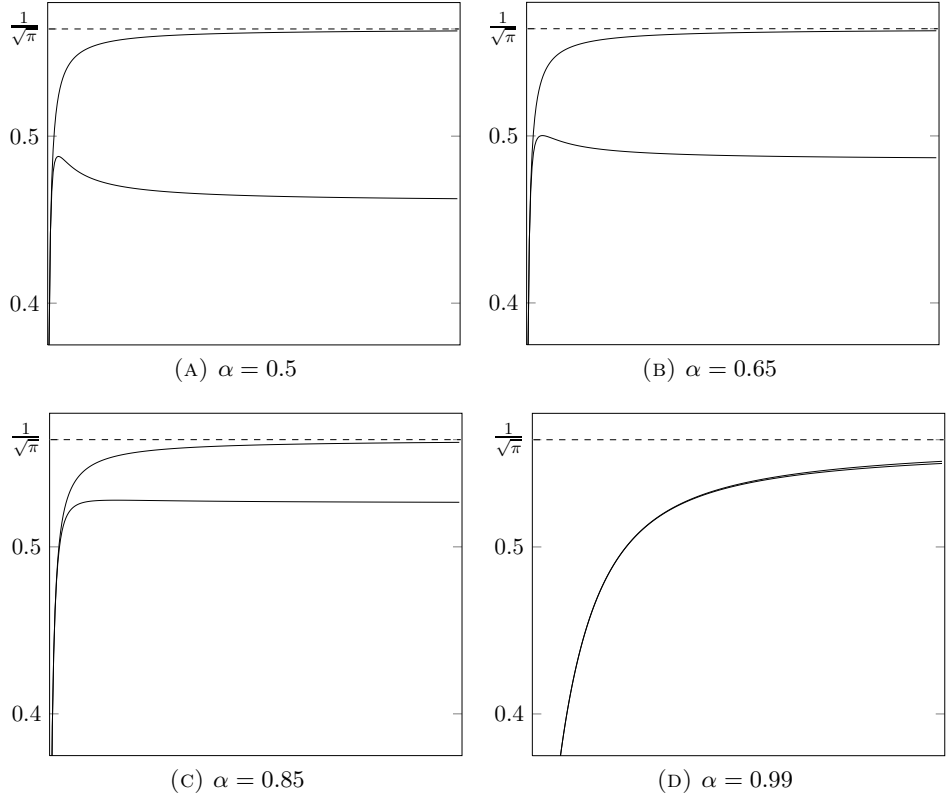

More formally,  using Proposition \ref{lemma:asymp} we get
 \begin{equation}\label{cndn}
0\leq  \tilde\kappa_n(\alpha) - \kappa_n(\alpha)\leq \mbox{$\sqrt{n\,\alpha(1\!-\!\alpha)}\;4n(1\!-\!\alpha)^2/\alpha = 4 n^{3/2}(1\!-\!\alpha)^{5/2}/\sqrt{\alpha}$}
 \end{equation}
while a formula for $\tilde \kappa_n(\alpha)$ involving the hypergeometric function 
$_2{}F_1(-n,\frac{1}{2};2;\cdot)$ obtained in \cite{csv}
gives explicitly (see the proof of \cite[Proposition 8]{csv})
  \begin{equation}\label{cn}
\tilde \kappa_n(\alpha) =\frac{1}{\pi} \int_{0}^{4n\alpha(1-\alpha)}\mbox{$\sqrt{\frac{1}{s} - \frac{1}{4n\alpha(1-\alpha)}}\left (1-\frac{s}{n}\right )^n$} ds.
  \end{equation}
Now, select a sequence $\theta_n\in\;]0,1[$ with $4n^{3/2}(1\!-\!\theta_{n})^{5/2}/\sqrt{\theta_n}\to 0$ and
$4n\theta_n(1\!-\!\theta_{n}) \to\infty$ ({\em e.g.} $\theta_n= 1\!-\! \ln n /n$).
Taking the constant $\alpha$ equal to $\theta_n$  it follows from \eqref{cndn} that $\tilde \kappa_n(\theta_n)-\kappa_n(\theta_n)\to 0$ 
and then \eqref{cn} gives
\begin{equation*}
  \lim_{n \to +\infty} \kappa_n(\theta_n)= \lim_{n \to +\infty} \tilde \kappa_n(\theta_n)=\mbox{$ \frac{1}{\pi}\Gamma \left (\frac{1}{2}\right )=\frac{1}{\sqrt{\pi}}$}
\end{equation*}
which completes the proof of Theorem \ref{thm:main}. \hspace{2ex}$\blacksquare$

\section{Final comments}\label{S4}

According to Theorem \ref{thm:main} the best uniform constant in \eqref{bnd} is $\kappa=1/\sqrt{\pi}$. Since this value 
is attained for $\alpha_n\equiv \alpha\sim 1$, it would be interesting to find the 
asymptotic regularity constant $\kappa=\kappa(\alpha)$ as a function of $\alpha$ for constant sequences
$\alpha_n\equiv\alpha$, and to discover which $\alpha$ yields the best convergence rate.
From Theorems \ref{P1} and \ref{teoint} we know that the recursive bounds $d_{mn}$
provide sharp estimates for arbitrary sequences 
$\alpha_n\in\;]0,1[$, so that  $\kappa(\alpha)=\sup_{n\in\NN} \kappa_n(\alpha)$.
Hence, defining 
\begin{equation}\label{gamma}
\gamma(\alpha)=\frac{\kappa(\alpha)}{\sqrt{\alpha(1\!-\!\alpha)}}=\sup_{n\in\NN} \sqrt{n}\,d_{n,n+1}(\alpha)/\alpha
\end{equation}
we get the convergence rate
$$\|x^n-Tx^n\|\leq\gamma(\alpha)/\sqrt{n}.$$

Using the inside-out algorithm in Remark \ref{inside-out} we computed $\gamma(\alpha)$
to obtain the  plot in Figure \ref{kappa}. This simple looking graph
masks the complexity of the analytic expression of $\gamma(\alpha)$ which involves the maximum of 
the terms $\sqrt{n}\,d_{n,n+1}(\alpha)/\alpha$. 
Note that the graph of $\gamma(\cdot)$ is slightly asymmetrical with respect to $\alpha=1/2$, which
is due to the fact that the optimal transports computed by the inside-out algorithm have a more complex 
structure for $\alpha<1/2$ than for $\alpha\geq 1/2$. 
In fact, from the inside-out algorithm it follows
that $d_{n,n+1}(\alpha)/\alpha$ is a polynomial for $\alpha\geq 0$ whereas it is piecewise polynomial for $\alpha<1/2$.
Hence $\gamma(\alpha)$ is formed by polynomial pieces which are glued together quite smoothly.

For the original iteration of Krasnosel'ski\v{\i} \cite{kra} with $\alpha=1/2$, the supremum 
in \eqref{gamma} seems to be attained at $n=8$ (see  Figure \ref{comp}{\sc a}). If this is the case,
then for $\alpha\geq 1/2$ the inside-out algorithm gives 
\begin{eqnarray*}
d_{8,9}(\alpha)/\alpha\!\!&=\!\!&1 - 8 \,\alpha + 64 \,\alpha^2 - 448 \,\alpha^3 + 2835 \,\alpha^4 - 16008 \,\alpha^5 + 79034 \,\alpha^6 - 334908 \,\alpha^7 + 1201873 \,\alpha^8\\
&&{}- 3622324 \,\alpha^9 + 9129380 \,\alpha^{10}- 19214722 \,\alpha^{11} +33796129 \,\alpha^{12} - 49776610 \,\alpha^{13} + 61566687 \,\alpha^{14} \\
&&{}- 64152608 \,\alpha^{15} \!+56488500 \,\alpha^{16}\! - 42133404 \,\alpha^{17} \!+ 26651679 \,\alpha^{18}\! - 14288252 \,\alpha^{19}\!+6472429 \,\alpha^{20}\\
&&{}- 2462126 \,\alpha^{21} + 778478 \,\alpha^{22} - 201354 \,\alpha^{23} + 41584 \,\alpha^{24} -6604 \,\alpha^{25} + 758 \,\alpha^{26} - 56 \,\alpha^{27} + 2 \,\alpha^{28},
\end{eqnarray*}
which would yield the following exact value for the sharp rate in Krasnosel'sk\v{\i}i's iteration 
$$\gamma(0.5)=\frac{46302245}{67108864}\sqrt{2}=\frac{5\cdot11\cdot 841859}{2^{26}}\sqrt{2}\sim 0.9757468.$$
Figure \ref{kappa} could give the impression that this rate is optimal. However more a careful 
look around $\alpha=1/2$ reveals that this is not
the case, and numerically we get smaller values such as $\gamma(0.48121)\sim0.974637$. 
The difference is quite small, though, and Karsnosel'ski\v{\i}'s original iteration seems close to 
optimal.

While throughout this paper we focused on finding the best uniform bounds 
of the form \eqref{bnd}, it would also be interesting to study the asymptotic rates 
by determining whether the limit $\kappa_\infty(\alpha)=\lim_{n\to\infty}\kappa_n(\alpha)$ exists. We computed 
$\kappa_n(\alpha)$ for different values of $\alpha$ and up to $n=50000$, which 
seems to provide a reasonable approximation for $\kappa_\infty(\alpha)$.
In particular, we found $\kappa_\infty(0.5)\sim\sqrt{2/3\pi}$ which is 
consistent with an observation in \cite{bb2}. A formal treatment of these questions 
remains open.

\begin{figure}[h]
\centering
\begin{tikzpicture}[scale=0.8]
    \begin{axis}[
        width  = 12cm,
        height = 8cm,
        table/header=false,
        table/row sep=\\,
        ymin=0,
        ymax=6,
        xmin=0,
        xmax=1,
        xtick={0,0.5,1},
        ytick={2,4,6},
        extra y ticks={0.9757468},  
       extra y tick labels={$\gamma(0.5)$}
    ]
 \addplot[dashed] coordinates {(0,0.9757468) (1,0.9757468)};
\addplot[color=black] table[mark=none,x expr=(\coordindex+1)*0.01,y index=0,y expr={\thisrowno{0}/sqrt(x*(1-x))}]{
0.5214545\\0.5199891\\0.5185473\\0.5171261\\0.5157268\\0.5143459\\0.5129954\\0.5116571\\0.5103467\\0.5090925\\0.5078599\\0.5066456\\0.5054385\\0.5042966\\0.5030907\\
0.5020848\\0.5009382\\0.4999543\\0.4989677\\0.4979039\\0.4970266\\0.4961949\\0.4952774\\0.4943857\\0.4936431\\0.4929464\\0.4922643\\0.4916175\\0.4910127\\0.4904743\\
0.4899118\\0.4893371\\0.4888574\\0.4883546\\0.4878638\\0.4875968\\0.4874033\\0.4872202\\0.4871228\\0.4870714\\0.4870223\\0.4869617\\0.4868875\\0.4867979\\0.4866919\\
0.4866612\\0.4867970\\0.4869538\\0.4873248\\0.4878734\\0.4884809\\0.4891106\\0.4897611\\0.4904432\\0.4912013\\0.4919776\\0.4927704\\0.4935782\\0.4944342\\0.4953359\\
0.4962496\\0.4971745\\0.4981708\\0.4991757\\0.5002019\\0.5012785\\0.5023597\\0.5034987\\0.5046460\\0.5058397\\0.5070576\\0.5083033\\0.5095841\\0.5109002\\0.5122479\\
0.5136331\\0.5150532\\0.5165085\\0.5180032\\0.5195401\\0.5211186\\0.5227444\\0.5244190\\0.5261470\\0.5279335\\0.5297849\\0.5317092\\0.5337168\\0.5358212\\0.5380407\\
0.540400593\\0.54289094\\0.545420492\\0.547985478\\0.550586058\\0.553223486\\0.555896467\\0.55861708\\0.56135809\\};
\end{axis}
\end{tikzpicture}
\caption{\label{kappa} The rate $\gamma(\alpha)$ as a function of $\alpha$ (for $0.01\leq\alpha\leq 0.99$).}
\end{figure}
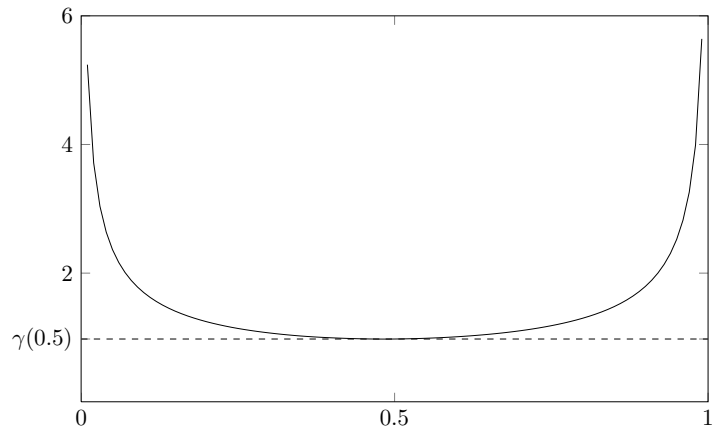

\pagebreak

\appendix

\section{Proof of Theorem \ref{t1}}
\label{appendixA}

\vspace{2ex}
\noindent (a) \underline{\em The function  $d(m,n)=d_{mn}$ defines a distance on the set $\mathcal N$.}

\begin{proof}

We already noted that $d_{mn}\geq 0$ and $d_{mn}=d_{nm}$ for all $n,m\in\mathcal N$, so we only have to 
prove the triangle inequality  and that $d_{mn}=0$ iff $m=n$.
We will show inductively that these properties hold for $m,n\leq\ell$ for each $\ell\in \mathcal N$. 
The base case $\ell=-1$ is trivial. Suppose that both properties hold for $\ell-1$ and let us prove them for $\ell$.

\vspace{1ex}

\noindent\underline{\em For $m,n\leq\ell$ we have $d_{mn}=0$ iff $m=n$}: 
From the induction hypothesis it suffices to prove that 
$d_{m\ell}>0$ when $m<\ell$ and $d_{\ell\ell}=0$. For $m<\ell$ take  an optimal 
transport plan $z\in F^{m\ell}$ so that $\sum_{i=0}^mz_{i\ell}=\pi_\ell^\ell=\alpha_\ell> 0$ and we find
$i\in\{0,\ldots,m\}$ with $z_{i\ell}>0$ and then $d_{m\ell}\geq z_{i\ell}d_{i-1,\ell-1}>0$. 
Now, for $m=\ell$ we consider the feasible transport plan $z\in F^{\ell\ell}$ with $z_{ii}=\pi_i^\ell$ and $z_{ij}=0$ for $j\neq i$
so that using the induction hypothesis we get
$$0\leq d_{\ell\ell}\leq \sum_{i=0}^\ell\sum_{j=0}^\ell z_{ij}d_{i-1,j-1}=\sum_{i=0}^\ell\pi_i^\ell d_{i-1,i-1}=0.$$

\vspace{1ex}

\noindent\underline{\em For $m,n,p\leq\ell$ we have $d_{mn}\leq d_{mp}+d_{pn}$}: Let $z^{mp}$ and $z^{pn}$ be optimal transport 
plans for $d_{mp}$ and $d_{pn}$ respectively, and define for $i=0,\ldots,m$ and $j=0,\ldots,n$
$$z_{ij}=\mbox{$\sum_{k=0}^p\frac{z^{mp}_{ik}z^{pn}_{kj}}{\pi_k^p}.$}$$
A straightforward computation shows that $z\in F^{mn}$ and therefore using the induction hypothesis we get
$$\begin{array}{ccl}
d_{mn}&\leq& \sum_{i=0}^m\sum_{j=0}^n z_{ij}d_{i-1,j-1}\\[1.5ex]
&=& \sum_{i=0}^m\sum_{j=0}^n\sum_{k=0}^p \frac{z^{mp}_{ik}z^{pn}_{kj}}{\pi_k^p}d_{i-1,j-1}\\[1.5ex]
&\leq& \sum_{i=0}^m\sum_{j=0}^n\sum_{k=0}^p \frac{z^{mp}_{ik}z^{pn}_{kj}}{\pi_k^p}[d_{i-1,k-1}+d_{k-1,j-1}]\\[1.5ex]
&=& \sum_{i=0}^m\sum_{k=0}^p \sum_{j=0}^n\frac{z^{mp}_{ik}z^{pn}_{kj}}{\pi_k^p}d_{i-1,k-1}+\sum_{j=0}^n\sum_{k=0}^p  \sum_{i=0}^m\frac{z^{mp}_{ik}z^{pn}_{kj}}{\pi_k^p}d_{k-1,j-1}\\[2.0ex]
&=& \sum_{i=0}^m\sum_{k=0}^p z^{mp}_{ik}d_{i-1,k-1}+\sum_{j=0}^n\sum_{k=0}^p  z^{pn}_{kj}d_{k-1,j-1}\\[1.5ex]
&=& d_{mp}+d_{pn}.
\end{array}
$$
\vspace{-4ex}

\ \hfill\end{proof}

\vspace{2ex}
\noindent (b) \underline{\em  For $m\leq n$ there exists a simple optimal transport plan $z$ for $d_{mn}$ with $z_{ii}=\pi_i^n$  for all $i=0,\ldots,m$.}
\begin{proof}  Let $z$ be an optimal solution for $d_{mn}$ and suppose that $z_{ii}<\pi_i^n$ for some $i\in\{0,\ldots,m\}$.
Choose the smallest of such $i$'s and take $j\in\{0,\ldots,m\}$ with $j\neq i$ and $z_{ji}>0$.
Now, take $k\in\{0,\ldots,n\}$ the smallest $k\neq i$ with $z_{ik}>0$ and
consider a modified transport plan $\tilde z$  identical to $z$ except for   (see Figure \ref{figr})
$$\begin{array}{l}
\tilde z_{ii}=z_{ii}+\varepsilon\\
\tilde z_{ji}=z_{ji}-\varepsilon\\
\tilde z_{ik}=z_{ik}-\varepsilon\\
\tilde z_{jk}=z_{jk}+\varepsilon.
\end{array}
$$
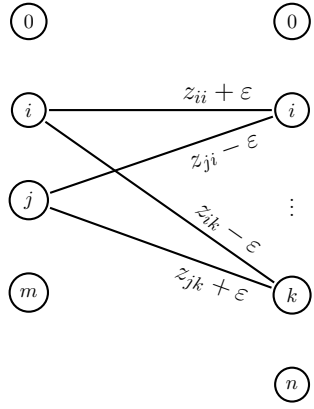
\begin{figure}
\centering
\begin{tikzpicture}[thick,-,shorten >= 1pt,shorten <= 1pt,scale=0.7,every node/.style={scale=0.7}]
\begin{scope}[start chain=going below,node distance=7mm]
 \node[on chain,draw,circle] (0)  {$0$};
 \node[on chain,draw, circle] (i)  {$i$};
 \node[on chain,draw, circle] (j)  {$j$};
\node[on chain,draw,circle] (m) {$m$};
\end{scope}
\begin{scope}[xshift=5cm,yshift=0cm,start chain=going below,node distance=7mm]
 \node[on chain,draw,circle] (00) {$0$};
 \node[on chain,draw,circle] (ii) {$i$};
 \node[on chain] (jj) {$\vdots$};
 \node[on chain, draw, circle] (kk)  {$k$};
 \node[on chain,draw,circle] (nn) {$n$};
\end{scope}

 edge[pil, bend right=45] (market.west)

\draw (i) -- (ii) node [above,near end] {\Large $z_{ii} +\varepsilon$};
\draw (i) -- (kk) node [above, near end, sloped] {\Large $z_{ik} -\varepsilon$};
\draw (j) -- (ii) node [below, near end, sloped] {\Large $z_{ji} -\varepsilon$} ;
\draw (j) -- (kk) node [below, near end, sloped] {\Large $z_{jk} +\varepsilon$};
\end{tikzpicture}
\caption{\label{figr}Redistribution of flow to construct a simple optimal transport}
\end{figure}

\noindent Using the triangle inequality we get the difference in transport cost as
$$C^{mn}(\tilde z)-C^{mn}(z)=\varepsilon(d_{i-1,i-1}+d_{j-1,k-1}-d_{j-1,i-1}-d_{i-1,k-1})\leq 0$$
so that taking $\varepsilon=\min\{z_{ji},z_{ik}\}$ we get a new optimal transport plan $\tilde z$ with 
either $\tilde z_{ik}=0$ or $\tilde z_{ji}=0$ (or both).

Updating $z\gets \tilde z$ and repeating this process iteratively we note that the pair $(i,k)$ increases lexicographically
and can remain constant for at most $m$ iterations. Indeed, for any given $i$ the same $k$ can 
occur only as long as $\tilde z_{ik}$ remains strictly positive, which implies that $\tilde z_{ji}=0$. 
This can occur in at most $m-1$ iterations (one of each $j\neq i$), after which we 
necessarily have either $\tilde z_{ik}=0$ and $k$ must increase in the next iteration, 
or $\tilde z_{ii}=\pi_i^n$, in which case $i$ increases in the next iteration.
It follows that  the process is finite and we eventually reach an optimal solution with 
$z_{ii}=\pi_i^n$ for all $i=0,\ldots,m$.
\end{proof}

\vspace{2ex}
\noindent (c) \underline{\em For fixed $m$ we have that $n\mapsto d_{mn}$ increases for $n\geq m$ and decreases for $n\leq m$. }
\begin{proof}  
Let us prove by induction that $d_{mn}\leq d_{m,n+1}$ for all $n\geq m$. The base case $n=m$ is trivial. 
For the induction step take $z\in F^{m,n+1}$, a simple optimal solution for $d_{m,n+1}$.
We observe that for $j=0,\ldots,n$ we have  $\sum_{i=0}^mz_{ij}=\pi_j^{n+1}=(1\!-\!\alpha_{n+1})\pi_j^n<\pi_j^n$ 
whereas $\sum_{i=0}^mz_{i,n+1}=\pi_{n+1}^{n+1}=\alpha_{n+1}$.
We will transform $z$ into a feasible solution for $d_{mn}$ by redirecting the 
inflow $\alpha_{n+1}$ of node $n+1$ towards all other nodes $j\leq n$ 
in such a way that the inflows $\sum_{i=0}^mz_{ij}$ are increased from $\pi_j^{n+1}$ to $\pi_j^n$, while 
maintaining  the flow balance at the source nodes $\sum_{j=0}^{n+1}z_{ij}=\pi_i^m$. In the process, we make 
sure that all flow shifts reduce the cost.

\begin{figure}[h!]
\centering
\begin{tikzpicture}[thick,-,shorten >= 1pt,shorten <= 1pt,scale=0.5,every node/.style={scale=0.7}]
\begin{scope}[start chain=going below,node distance=6mm]
 \node[on chain,draw,circle] (i0)  {$0$};
 \node[on chain,draw, circle] (i1)  {$i$};
\node[on chain,draw,circle] (im1) {$m$};
\end{scope}
\begin{scope}[xshift=5cm,yshift=0cm,start chain=going below,node distance=6mm]
 \node[on chain,draw,circle] (j0) {$0$};
 \node[on chain,draw, circle] (j1)   {$i$};
\node[on chain,draw,circle] (jm1) {$m$};
 \node[on chain, draw, circle] (jaux) {$k$};
 \node[on chain,draw,circle, scale=0.6] (j2) {$n+1$};
\end{scope}
\draw (i1) -- (jaux)  node [below, midway, sloped] {\Large $z_{ik} -\varepsilon$} ;
\draw (i1) -- (j1) node [above, midway] {\Large $z_{ii} +\varepsilon$} ;

\begin{scope}[xshift=10cm,start chain=going below,node distance=6mm]
 \node[on chain,draw,circle] (ii0)  {$0$};
 \node[on chain,draw, circle] (ii1)  {$i$};
\node[on chain,draw,circle] (iim1) {$m$};
\end{scope}
\begin{scope}[xshift=15cm,yshift=0cm,start chain=going below,node distance=6mm]
 \node[on chain,draw,circle] (jj0) {$0$};
 \node[on chain,draw, circle] (jj1)   {$i$};
\node[on chain,draw,circle] (jjm1) {$m$};
 \node[on chain, draw, circle] (jjaux) {$j$};
 \node[on chain,draw,circle, scale=0.6] (jj2) {$n+1$};
\end{scope}
\draw (ii1) -- (jjaux)  node [above, midway,sloped]{\Large $z_{ij} +\varepsilon$};
\draw (ii1) -- (jj2) node [below, midway, sloped] {\Large $z_{i,n+1} -\varepsilon$} ;

\begin{scope}[xshift=20cm,start chain=going below,node distance=6mm]
 \node[on chain,draw,circle] (ii0)  {$0$};
 \node[on chain] (idost)   {$\vdots$};
\node[on chain,draw,circle] (iin1) {$n$};
 \node[on chain,draw, circle] (ii1) {$i$};
\node[on chain,draw,circle] (iim1) {$m$};
\end{scope}
\begin{scope}[xshift=25cm,yshift=0cm,start chain=going below,node distance=6mm]
 \node[on chain,draw,circle] (jj0) {$0$};
 \node[on chain, draw, circle] (jjaux) {$j$};
 \node[on chain,draw,circle] (jj2) {$n$};
\end{scope}
\draw (ii1) -- (jjaux)  node [above, midway,sloped]{\Large $z_{ij} -\varepsilon$};
\draw (ii1) -- (jj2) node [below, midway, sloped] {\Large $z_{in} +\varepsilon$} ;

\end{tikzpicture}
\caption{\label{figflow} Redistribution of flows to prove monotonicity.} 
\end{figure}
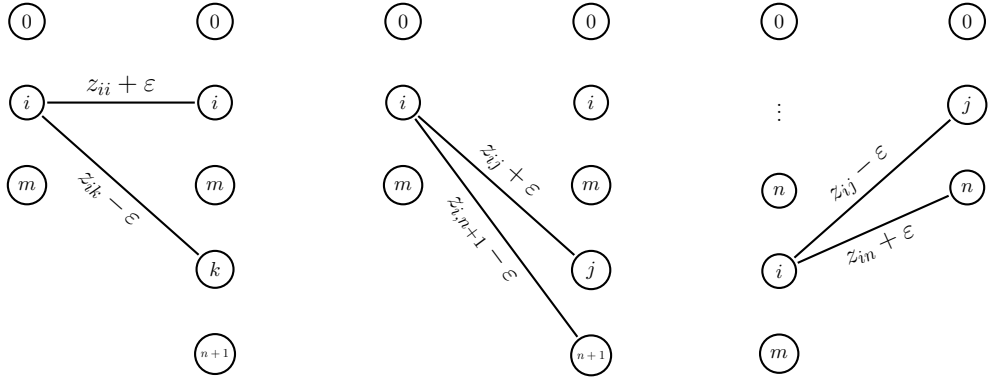

We proceed in two stages.  Firstly, for each $i=0,\ldots,m$ we choose any $k>m$ with $z_{ik}>0$
from which we remove some amount $\varepsilon$ and re-route it to destination $i$ by augmenting 
$z_{ii}$, as shown in Figure \ref{figflow} left. This increases the inflow to destination node $i$ while 
keeping the flow balance at the source node $i$, and the cost is reduced by $\varepsilon(d_{i-1,i-1}-d_{i-1,k-1})<0$. 
 We repeat this process until the inflow to $i$ reaches $\pi_i^n$ as required. In the second stage, for each node $j=m+1,\ldots,n$ we 
 increase its inflow up to $\pi_j^n$ by reducing  the remaining flows $z_{i,n+1}>0$ for $i=0,\ldots,m$ 
and augmenting $z_{ij}$, as shown in Figure \ref{figflow} center. This changes the transportation cost by $\epsilon(d_{i-1,j-1}-d_{i-1,n})$.
Since $i-1\leq j-1\leq n$ the induction hypothesis implies that this cost difference is again negative.
After completing these flow transfers we have $\sum_{i=0}^mz_{ij}=\pi_j^n$ for all $j=0,\ldots,n$
and a fortiori $\sum_{i=0}^mz_{i,n+1}=0$. Hence
the transformed $z$ is feasible for $d_{mn}$ and therefore $d_{mn}\leq C^{mn}(z)$.
Since along the transformation the cost was reduced from its initial value $d_{m,n+1}$, 
we conclude $d_{mn}\leq d_{m,n+1}$, completing the induction step.

\vspace{1ex}
The proof of  $d_{m,n-1}\geq d_{mn}$ for $n\leq m$ is  similar. The base case $n=0$ is trivial since
 $d_{m,-1}=1\geq d_{m0}$. For the induction step, take $z\in F^{m,n-1}$ a simple optimal solution 
 for $d_{m,n-1}$. Let us transform $z$ into a feasible solution for $d_{mn}$ while reducing the cost.
 To this end, for each $j=0,\ldots,n-1$ we take some positive flow $z_{ij}>0$ with $i\geq n$ from which we remove
 a positive amount $\varepsilon$ and reroute it to node $n$ by increasing $z_{in}$, as shown in Figure \ref{figflow} right.
As a result, the cost changes by $\varepsilon(d_{i-1,n-1}-d_{i-1,j-1})$, which is negative by the induction hypothesis
 since $j-1\leq n-1\leq i-1$. We repeat these flow transfers until the inflow $\sum_{i=0}^mz_{ij}$ is reduced from 
 $\pi_j^{n-1}$ to $\pi_j^n$.  The difference $\pi_j^{n-1}-\pi_j^n=\alpha_n\pi_j^{n-1}$ is now redirected to node $n$ and
 therefore $\sum_{i=0}^mz_{in}=\sum_{j=0}^{n-1} \alpha_n\pi_j^{n-1}=\alpha_n=\pi_n^n$. It follows that the transformed flow
 is feasible for $d_{mn}$ and therefore $d_{mn}\leq C^{mn}(z)\leq d_{m,n-1}$.
\end{proof}

In order to establish the 4-point inequality (d) we will exploit inductively 
the fact that the inside-out algorithm in Remark \ref{inside-out} computes 
an optimal transport. More formally, we will use the following

\begin{lemma}\label{prop1}
Suppose that $\alpha_n\geq \frac{1}{2}$ for all $n\in\NN$. Let $n\in \NN$ and assume that the 4-point inequality 
holds for all $0\leq i\leq k\leq j\leq l<n$. Then for all $m\leq n$ the inside-out procedure
computes an optimal transport $z=z^{mn}$ given by 
$$\begin{array}{ccll}
z_{ii}&=&\pi_i^n& \mbox{ for } i = 0, \ldots, m,\\
z_{mj}&=&\pi_j^n &\mbox{ for } j= m+1, \ldots, n-1,\\
z_{in}&=&\pi_i^m - \pi_i^{n}& \mbox{ for } i = 0, \ldots, m-1,\\
z_{mn}&=&\pi_m^m-\sum_{j=m}^{n-1} \pi_j^n.& 
\end{array}$$
\end{lemma}
\begin{proof} 
Let $0\leq m\leq n$. Since  the costs  in $(\mathcal P_{mn})$  are the distances $d_{i-1,j-1}$ 
with the nodes shifted by $-1$, we have the 4-point inequalities required to argue, as in
Remark \ref{inside-out}, that there is an optimal transport  in which the flows do not cross, 
which is precisely the one computed by the inside-out algorithm.
Now, this procedure sets $z_{ii}=\pi_i^n$ for $i=0,\ldots,m$.
Next, we note that $z_{mj}=\pi_j^n$ holds for all $j= m, \ldots, n-1$  if and only if the supply
of node $m$ suffices to satisfy the demands of all nodes $j=m, \ldots, n-1$, that is,
$\pi_m^m \geq \sum \nolimits_{j=m}^{n-1}\pi_j^n$. This is precisely the case, since $\pi_m^m=\alpha_m\geq\frac{1}{2}$
while $\sum \nolimits_{j=m}^{n-1}\pi_j^n\leq 1-\pi_n^n=1-\alpha_n\leq\frac{1}{2}$.
Hence we have $z_{mj}=\pi_j^n$ for $j= m+1, \ldots, n-1$ and thus the remaining supplies
are sent to node $n$, namely $z_{in}=\pi_i^m - \pi_i^{n}$  for $i = 0, \ldots, m-1$
and $z_{mn}=\pi_m^m-\sum_{j=m}^{n-1}\pi_j^n$.
 \end{proof}

 \begin{remark}
In the constant case $\alpha_n\equiv \alpha$ the condition $\alpha\geq\frac{1}{2}$ is not only sufficient but
also necessary for the optimal transport to be as decribed in the Lemma. Indeed, in the constant case the 
condition $\pi_m^m \geq \sum \nolimits_{j=m}^{n-1}\pi_j^n$ is equivalent to 
$\alpha\geq\beta-\beta^{n-m+1}$ which holds for all $0\leq m \leq n$
iff $\alpha\geq\frac{1}{2}$.
 \end{remark}

\vspace{2ex}
\noindent
We now proceed to establish the 4-point inequality.
\vspace{2ex}

\noindent (d)  \underline {If $ \alpha_n\geq 1/2$ for all $n \in \mathbb N$, then for all integers $0\leq i\leq k\leq j\leq l$ we have $d_{il}+d_{kj}\leq d_{ij}+d_{kl}$.}

\begin{proof}
For $k=j$ this is the triangle inequality which was already proved, so let us consider the case $k<j$.
Let us fix $i \in \NN$ and let us prove by induction that for all $l > i$ we have
 \begin{equation} \label{4points} \tag{$I_l$}
 d_{il}+d_{kj}\leq d_{ij}+d_{kl}  \quad \text{ for all }  i\leq k\leq j\leq l.
 \end{equation}
The base case $l=i+1$ reads $ d_{i,i+1}+d_{kj}\leq d_{ij}+d_{k,i+1}$ with $k=i$ and $j=i+1$,
which holds trivially. For the induction step, we will exploit the following equivalence.
\begin{claim}\label{claim} For each $\ell\in\NN$ the following are equivalent:
\begin{itemize}
\item[(a)] For all integers $0\leq i\leq k < j\leq l\leq\ell$ we have $d_{il}+d_{kj}\leq d_{ij}+d_{kl}$.
\item[(b)] For all integers $0\leq m<n<\ell$ we have $  d_{m,n+1} + d_{m+1,n} \leq d_{m,n} + d_{m+1,n+1}$.
\end{itemize}
\end{claim}
\noindent{\em Proof of Claim \ref{claim}.} Clearly (b) is the special case of (a) with $i\!=\!m$, $k\!=\!m\!+\!1$, $j\!=\!n$, $l\!=\!n\!+\!1$.
In order to show that conversely (b) implies (a), take $0\leq i\leq k < j\leq l \leq \ell$. By telescoping we have 
$d_{il}-d_{ij}=\sum_{n=j}^{l-1} d_{i,n+1}-d_{in}$ and since from
(b) we have that the difference $d_{i,n+1}-d_{in}$ increases with $i$, we deduce
$d_{il}-d_{ij}\leq \sum_{n=j}^{l-1} d_{k,n+1}-d_{kn}=d_{kl}-d_{kj}$ which yields (a).
$\blacksquare$
  \vspace{2ex}

\noindent{\underline{Induction step:}} Let us assume that \eqref{4points} holds.  
In view of Claim \ref{claim}, in order to prove ($I_{l+1}$) it suffices to show that 
for $n=l+1$ one has
\[
\Delta:=d_{m,n}+d_{m+1,n+1}-d_{m,n+1}-d_{m+1,n} \geq 0.
\]
Since $\alpha_n\geq \frac{1}{2}$  and thanks to the induction hypotesis $(I_l)$,
we can use Lemma \ref{prop1} to obtain the following explicit expressions for each of the 
4 terms appearing in $\Delta$: 

\begin{equation*}
\begin{aligned}
d_{m,n}&=\sum_{j=m+1}^{n-1}\pi_j^nd_{m-1,j-1}+\sum_{i=0}^{m-1}(\pi_i^m\!-\!\pi_i^n)d_{i-1,n-1}+(\pi_m^m\!-\!\sum_{j=m}^{n-1}\pi_j^n)d_{m-1,n-1}\\
&=\sum_{j=m+1}^{n}\pi_j^n[d_{m-1,j-1}-d_{m-1,n-1}]+\sum_{i=0}^m(\pi_i^m\!-\!\pi_i^n)d_{i-1,n-1},\\[1ex]
d_{m+1,n+1}&=\sum_{j=m+2}^{n}\pi_j^{n+1}[d_{m,j-1}-d_{m,n}]+\sum_{i=0}^{m+1}(\pi_i^{m+1}\!-\!\pi_i^{n+1})d_{i-1,n}\\
&=\sum_{j=m+2}^{n}\pi_j^{n+1}[d_{m,j-1}-d_{m,n-1}]+\sum_{i=0}^{m+1}(\pi_i^{m+1}\!-\!\pi_i^{n+1})d_{i-1,n}+\sum_{j=m+2}^{n}\pi_j^{n+1}[d_{m,n-1}-d_{m,n}]\\
&=\!\!\sum_{j=m+1}^{n}\pi_j^{n+1}[d_{m,j-1}-d_{m,n-1}]+\sum_{i=0}^{m}(\pi_i^{m+1}\!\!-\!\pi_i^{n+1})d_{i-1,n}+\sum_{j=m+1}^{n}\pi_j^{n+1}[d_{m,n-1}-d_{m,n}]+\pi_{m+1}^{m+1}d_{m,n},\\[1ex]
\end{aligned}
\end{equation*}

\begin{equation*}
\begin{aligned}
d_{m,n+1}&=\sum_{j=m+1}^{n}\pi_j^{n+1}[d_{m-1,j-1}-d_{m-1,n}]+\sum_{i=0}^m(\pi_i^m\!-\!\pi_i^{n+1})d_{i-1,n}\\
&=\sum_{j=m+1}^{n}\pi_j^{n+1}[d_{m-1,j-1}-d_{m-1,n-1}]+\sum_{i=0}^m(\pi_i^m\!-\!\pi_i^{n+1})d_{i-1,n}+\sum_{j=m+1}^{n}\pi_j^{n+1}[d_{m-1,n-1}-d_{m-1,n}],\\[1ex]
d_{m+1,n}&=\sum_{j=m+2}^{n-1}\pi_j^n[d_{m,j-1}-d_{m,n-1}]+\sum_{i=0}^{m+1}(\pi_i^{m+1}\!-\!\pi_i^n)d_{i-1,n-1}\\
&=\sum_{j=m+1}^{n}\pi_j^n[d_{m,j-1}-d_{m,n-1}]+\sum_{i=0}^{m}(\pi_i^{m+1}\!-\!\pi_i^n)d_{i-1,n-1}+\pi_{m+1}^{m+1}d_{m,n-1}.
\end{aligned}
\end{equation*}
Replacing these expressions in $\Delta$ and rearranging terms we can rewrite $\Delta=A+B+C$ with
\begin{equation*}
\begin{aligned}
A&=\sum_{j=m+1}^{n}(\pi_j^n-\pi_j^{n+1})[d_{m-1,j-1}-d_{m-1,n-1}-d_{m,j-1}+d_{m,n-1}]\geq 0\\
B&=\sum_{i=0}^m (\pi_i^m\!-\!\pi_i^{m+1})[d_{i-1,n-1} -d_{i-1,n}+d_{m-1,n}-d_{m-1,n-1}]\geq 0\\
C&=\sum_{i=0}^m (\pi_i^m\!-\!\pi_i^{m+1})[d_{m-1,n-1}-d_{m-1,n}]+\sum_{j=m+1}^{n}\pi_j^{n+1}[d_{m,n-1}-d_{m,n}]+\pi_{m+1}^{m+1}d_{m,n}\\
&{\ \ \ \ }-\sum_{j=m+1}^{n}\pi_j^{n+1}[d_{m-1,n-1}-d_{m-1,n}]-\pi_{m+1}^{m+1}d_{m,n-1}\\
&=(\pi_{m+1}^{m+1}-\sum_{j=m+1}^{n}\pi_j^{n+1})[d_{m-1,n-1}-d_{m-1,n}-d_{m,n-1}+d_{m,n}]\geq 0,
\end{aligned}
\end{equation*}
where the positivity of the three quantities $A, B, C$ follows from the induction hypothesis.
\end{proof}


\section{Proof of Proposition \ref{lemma:asymp}}\label{prob}

\noindent
\underline{For $\alpha_n \equiv \alpha \in [\frac{1}{2},1[$ we have $0\leq c_{mn}-d_{mn}\leq 4m(1-\alpha)^2$.}

\begin{proof}
Starting from the estimate \eqref{infnorm}, a simple application of the Mean Value Theorem yields
\begin{equation}\label{fin}
|c_{mn}-d_{mn}|\leq m |\!|\!|\mathbf D-\mathbf C|\!|\!|_\infty
\end{equation}
so it suffices to prove that $ |\!|\!|\mathbf D - \mathbf C |\!|\!|_{\infty} \leq 4(1\!-\!\alpha)^2$.
Hence, denoting $\beta = 1- \alpha$ we must show that for each $s\in\mathcal S$ the sum 
$\Gamma(s)=\sum_{s'\in\mathcal S}| \mathbf D(s,s')-\mathbf C(s,s')|$ is at most $4\beta^2$.

For the absorbing states we have $\Gamma(f)=\Gamma(h)=0$ so we just consider
a transient state $s=mn\in\mathcal S$.
For $s'=f$ we have
$\mathbf D(mn,f) =\sum_{i=0}^m \pi_i^n= \mathbf C(mn,f)$, whereas
for $s'=h$ we have
\begin{equation*}
  \mathbf D(mn,h)= \pi_0^m -\pi_0^n=\pi_0^m\sum_{j=m+1}^n \pi_j^n= \mathbf C(mn,h). 
\end{equation*}
It follows that in the sum $\Gamma(mn)$ we can ignore the terms $s'=f,h$ so
\begin{equation}\label{sumita}
\Gamma(mn)=\sum_{kl\in\mathcal S}| \mathbf D(mn,kl)-\mathbf C(mn,kl)|=\sum_{1\leq i\leq m<j\leq n}|z^{mn}_{ij}-\tilde z^{mn}_{ij}|
\end{equation}
where $z^{mn}$ is the optimal transport described in Lemma \ref{prop1} and $\tilde z^{mn}$ is given by \eqref{ztilde}.

Using the expressions of $z^{mn}$ and $\tilde z^{mn}$ we get explicitly
  \begin{equation*}
     \vert z^{mn}_{ij} -  \tilde z^{mn}_{ij} \vert = 
\begin{cases}
  \alpha^2 \beta^{m-i+n-j} & \text{ if } i=1,\ldots,m-1 \text { and } j=m+1, \ldots,n-1,\\
\alpha \beta^{m-i+1}(1-\beta^{n-m-1}) & \text{ if } i=1,\ldots,m-1 \text { and } j=n,\\
\alpha \beta^{n-j+1} &\text{ if } i=m \text { and } j=m+1,\ldots,n-1,\\
\beta^2 (1-\beta^{n-m-1}) &\text{ if } i=m \text { and } j=n,\\
0 & \text{ otherwise}
\end{cases}
  \end{equation*}
which replaced into \eqref{sumita} yield the conclusion
\begin{equation*}
\Gamma(mn) =\alpha^2  \sum_{i=1}^{m-1}\sum_{j=m+1}^{n-1}\beta^{m-i+n -j} +\alpha (1\!-\!\beta^{n-m-1})\sum_{i=1}^{m-1} \beta^{m-i+1} 
+ \alpha \sum_{j=m+1}^{n-1}\beta^{n-j+1} + \beta^2 (1\!-\!\beta^{n-m-1})\leq 4\beta^2.
\end{equation*}
\vspace{-4ex}

\ \hfill \end{proof}

\end{document}